\documentclass[a4paper]{amsart}

\usepackage[OT2,T1]{fontenc}

\usepackage{amssymb}	
\usepackage{wasysym}
\usepackage{verbatim}	
\usepackage{mathrsfs}	
\usepackage{dsfont}
\usepackage[all]{xy}
\usepackage{indentfirst}
\usepackage{cancel}

\theoremstyle{plain}

\newtheorem{thm}{Theorem}[section]
\newtheorem{lem}[thm]{Lemma}
\newtheorem{prop}[thm]{Proposition}
\newtheorem{coro}[thm]{Corollary}

\theoremstyle{definition}

\newtheorem{defn}[thm]{Definition}

\newtheorem*{nota}{Notation}

\theoremstyle{remark}
\newtheorem{rem}[thm]{Remark}

\newcommand{\on}{\operatorname}

\newcommand{\mc}{\mathcal}
\newcommand{\mb}{\mathbf}
\newcommand{\mbb}{\mathbb}
\newcommand{\mr}{\mathrm}
\newcommand{\mf}{\mathfrak}

\newcommand{\mscr}{\mathscr}

\newcommand{\id}{\ensuremath{\mathop{\rm id\,}\nolimits}}

\newcommand{\Z}{\mathbb{Z}}
\newcommand{\R}{\mathbb{R}}

\newcommand{\Q}{\mathbb{Q}}

\newcommand{\A}{\mathbb{A}}

\newcommand{\Gm}{\mbb{G}_m}

\newcommand{\m}{\mc{M}}
\newcommand{\vp}{\varphi}
\newcommand{\ve}{\varepsilon}
\newcommand{\om}{\omega}
\newcommand{\Om}{\Omega}

\newcommand{\ol}{\overline}

\newcommand{\st}{\stackrel}
\newcommand{\mx}{\mbox}

\newcommand{\geqs}{\geqslant}
\newcommand{\leqs}{\leqslant}
\newcommand{\ot}{\otimes}

\newcommand{\w}{\wedge}
\newcommand{\wh}{\widehat}
\newcommand{\wt}{\widetilde}

\newcommand{\sm}{\setminus}
\newcommand{\sha}{\mbox{ {\footnotesize {\fontencoding{OT2}\selectfont\char88} }} }

\newcommand{\dd}{\on{d}\!}

\newcommand{\ra}{\rightarrow}
\newcommand{\lra}{\longrightarrow}

\newcommand{\Gr}{\on{Gr}}
\newcommand{\HH}{\on{H}}

\newcommand{\p}{\mathbb{P}}

\newcommand{\Fc}{\mc{F}}

\newcommand{\dme}[1][]{\mc{DM}_{gm}^{eff}}
\newcommand{\dm}[1][]{\mc{DM}_{gm}}

\newcommand{\SmCorr}[1][]{\on{\SmCorr}_{#1}}

\newcommand{\h}{\mf{H}}

\newcommand{\equ}{\begin{equation}}
\newcommand{\eq}{\equ}
\newcommand{\eqf}{\end{equation}}
\newcommand{\fr}[2]{\frac{#1}{#2}}
\newcommand{\ds}{\displaystyle}
\def\eqa{\begin{eqnarray*}}
\def\eqn{\begin{eqnarray}}
\def\eqaf{\end{eqnarray*}}
\def\eqnf{\end{eqnarray}}

\newcommand{\f}[2]{f_{#1_{1},\ldots,#1_{#2}}}
\newcommand{\Inc}[1]{\int_{[0,1]^{#1}}}
\newcommand{\Intd}[1]{\int_{\Delta_{#1}}}
\newcommand{\ent}[1]{[\![#1]\!]}
\DeclareMathOperator{\Prod}{{\mbox{\small $\prod$}}}
\newcommand{\yks}[1]{\mathbf{x}(\mathbf{k},{#1})}
\newcommand{\yk}[1]{\Prod\mathbf{x}(\mathbf{k},{#1})}
\newcommand{\ysk}[1]{\Prod\mathbf{x}(\mathbf{k},\leqs{#1})}
\newcommand{\ykz}[1][0]{\mathbf{x}_{#1}}
\newcommand{\pykz}[1][0]{\Prod \mathbf{x}_{#1}}
\newcommand{\pyk}{\Prod \mathbf{x}}

\newcommand{\yls}[1]{\mathbf{x'}(\mathbf{l},{#1})}

\newcommand{\ysl}[1]{\Prod\mathbf{x'}(\mathbf{l},\leqs{#1})}
\newcommand{\ylz}[1][0]{\mathbf{x}_{#1}\mathbf{'}}
\newcommand{\pylz}[1][0]{\Prod \mathbf{x'}_{#1}}
\newcommand{\pyl}{\Prod \mathbf{x'}}

\newcommand{\dx}[1][]{{d^{#1}x}}
\newcommand{\du}[1][]{{d^{#1}u}}
\newcommand{\mn}[1]{\m_{0,#1+3}}
\newcommand{\mnb}[1]{\ol{\m_{0,#1+3}}}
\newcommand{\oms}[1]{\om_{\mathbf{#1}}}
\newcommand{\As}[1]{A_{\mathbf{#1}}}

\newcommand{\ab}[2]{B_{#1}^{A_{#2}}}
\newcommand{\abs}[2]{\ab{#1}{\mathbf{#2}}}
\newcommand{\pa}[1]{\left(#1 \right)}
\newcommand{\PrJs}[3]{\mr{\Pi}^{#1_{#2}}\mb{#3}}
\newcommand{\pjs}[2][s]{\PrJs{J}{#2}{#1}}

\title{Motivic double shuffle }
\author{Ismael Soud{\`e}res}
\thanks{this work has been partially supported by a Marie Curie Early
  Stage Training fellowship for the AAG network at Durham University.}
\address{Institut de Mathématiques de Jussieu (IMJ), Université Paris
  Diderot -- Paris 7, 175 rue du Chevaleret, 75013 Paris}
\date{\today}

\begin{document}

\maketitle
\tableofcontents

\section*{Introduction}
For a $p$-tuple $\mathbf{k} = (k_1,\ldots , k_p)$ of positive integers and
$k_1 \geqs 2$, the  multiple zeta value $\zeta(\mathbf{k})$ is defined as
\[\zeta(\mathbf{k}) = \sum_{n_1>\ldots>n_p>0} \frac{1}{n_1^{k_1}\cdots
  n_p^{k_p}}.\] 
These values satisfy two families of algebraic (quadratic) relations
known as double shuffle relations, or shuffle and stuffle, described below.

In \cite{MZMSGM} A.B. Goncharov and Y. Manin define a motivic version of
multiple zeta values using certain framed mixed Tate motives attached
to moduli spaces of genus $0$ curves. In this context, the real
multiple zeta values appear naturally as periods of those motives attached to
the moduli
spaces of curves. They do not prove the double shuffle relations directly,
referring instead to previous work by A.B. Goncharov in which, using a
different definition of motivic multiple polylogarithms based on
$(\p^1)^n$ rather than moduli spaces, the motivic double shuffle
relations are shown via results on variations of mixed Hodge structure.

The goal of this article is to give an elementary proof of the double
shuffle relations directly for the Goncharov and Manin motivic multiple
zeta values. The shuffle relation (Proposition \ref{motivicshu}) is
straightforward, but for the
stuffle (Proposition \ref{stm}) we use a modification of a method
first introduced by P. Cartier for the purpose of proving
stuffle for the real multiple zeta values via integrals and blow-up sequences.
In this article, we will work over the base field $\Q$.

\section{Integral representation of the double shuffle relations}
\subsection{Series representation of the stuffle relations}\label{serstusec}

The stuffle product of a $p$-tuple $\mathbf{k}=(k_1,\ldots , k_p)$ and a
$q$-tuple $\mathbf{l}=(l_1,\ldots ,l_q)$ is defined recursively by the
formula: 
\begin{multline} \label{sturec}\mathbf{k}*\mathbf{l}=(\mathbf{k}*(l_1,\ldots
,l_{q-1}))\cdot l_q +((k_1,\ldots ,k_{p-1})*\mathbf{l})\cdot k_p \\ + (
(k_1,\ldots , k_{p-1})*(l_1,\ldots , l_{q-1}))\cdot (k_p + l_q)\end{multline}
and $\mathbf{k}*()=()*\mathbf{k}=\mathbf{k}$. Here the $+$ is a formal sum,
$A\cdot
a$ means that we concatenate $a$ at the end of the tuple $A$ and
$\cdot$ is linear in $A$.

Let $\mathbf{k}$ and $\mathbf{l}$ be two such tuples of integers. 
We will write $\on{st}(\mb k, \mb l)$ for the set of individual terms in
the formal sum $\mathbf{k}*\mathbf{l}$ whose coefficients are
all equal to $1$. Such a generic term is then denoted by $\sigma \in
\on{st}(\mb k, \mb l)$.

In order
to multiply two multiple zeta values $\zeta(\mathbf{k})$ and
$\zeta(\mathbf{l})$,
we split the summation domain of the product
$\zeta(\mathbf{k})\zeta(\mathbf{l})$
\[\left\{0 < n_1 <\ldots < n_p\right\} \times \left\{0 < m_1 <\ldots<
  m_q\right\}\] 
into all the domains that preserve the order of the $n_i$ as well as the order
of the
$m_j$ and into the boundary domains where some $n_i$ are equal
to some $m_j$. We obtain the following well-known proposition, giving
the quadratic relations (\ref{serstu}) between multiple zeta
values known as the \emph{stuffle relations}:
\begin{prop} \label{serstuprop}Let $\mathbf{k}=(k_1,\ldots ,k_p)$ and
  $\mathbf{l}=(l_1,\ldots ,l_q)$ as above with $k_1$, $l_1 \geqs 2$. Then
  we have:
\eq\label{serstu}\zeta(\mathbf{k})\zeta(\mathbf{l})=\left(\sum_{n_1>\ldots
    >n_p>0}\frac{1}{n^{k_1}_1\cdots n^{k_p}_p}\right)\left(
    \sum_{m_1>\ldots >m_q>0}\frac{1}{m^{l_1}_1\cdots
      m^{l_q}_q}\right)=\sum_{\sigma \in \on{st}(\mathbf{k},\mathbf{l})}
\zeta(\sigma).\eqf
\end{prop}
\subsection{Integral representation of the shuffle relations}
To the tuple $\mathbf{k}$, with $n=k_1 +\cdots + k_p$, we associate the
$n$-tuple:
\[\ol{k}=(\underbrace{0,\ldots , 0}_{k_1-1\mx{ times}}, 1,\ldots
  ,\underbrace{0,\ldots ,0}_{k_p-1\mx{ times}},
  1)=(\ve_n,\ldots,\ve_1)\]
  and the differential form, introduced by Kontsevich
\begin{align} \label{Kontform}
 \om_{\mathbf{k}}=\om_{\ol{k}}= (-1)^p\fr{dt_1}{t_1-\ve_1}\w 
\cdots\w \fr{dt_n}{t_n-\ve_n}.
\end{align}

Then, setting  $\Delta_n=\{0 <t_1 <\ldots <t_n< 1\}$, direct
integration yields:
\[\zeta(\mathbf{k})=\Intd{n}\om_{\mathbf{k}}.\]
The shuffle product of an $n$-tuple $(e_1,\ldots ,e_n)=e_1
\cdot \ol{e}$ and an $m$-tuple $(f_1,\ldots ,f_m)=f_1 \cdot
\ol{f}$
is defined recursively by:
\eq(e_1,\ldots ,e_n)\sha(f_1,\ldots ,f_m)=
e_1\cdot(\ol{e}\sha 
(f_1\cdot \ol{f})) + f_1\cdot((e_1\cdot \ol{e})\sha 
\ol{f})\label{shurec}\eqf
and $\ol{e}\sha()=()\sha\ol{e} = \ol{e}$. Here, as above, the $+$ is a formal
sum, $b\cdot B$ means that we concatenate $b$ at the beginning of the tuple
$B$ and $\cdot$ is linear in $B$.

Let $\mathbf{k}$ and $\mathbf{l}$ be two tuples of integers as above. 
We will write $\on{sh}(\ol k, \ol l)$ for the set of the individual terms in
the formal sum $\ol{k}\sha\ol{l}$ whose coefficients are
all equal to $1$. Such a generic term is then denoted by $\sigma \in
\on{sh}(\ol k, \ol l)$ and can be identified with a unique permutation
$\tilde{\sigma}$ of $\{1,\ldots n+m\}$ such that
$\tilde{\sigma}(1)<\ldots<\tilde{\sigma}(n)$ and
$\tilde{\sigma}(n+1)<\ldots<\tilde{\sigma}(n+m)$. The permutation
$\tilde{\sigma}$ will simply be denoted by $\sigma$ when the context will be
clear enough.

We will put an index $\sigma$ on any object which naturally depends on
a shuffle. The following proposition yields the
quadratic relations (\ref{shint}) known as the \emph{shuffle relations}.
\begin{prop}\label{shintprop}Let $\mathbf{k}= (k_1,\ldots ,k_p)$ and
  $\mathbf{l}=(l_1,\ldots ,l_q)$ with $k_1$, $l_1 \geqs 2$. Then:
\eq
\label{shint}\Intd{n}\om_{\ol{k}}\Intd{m}\om_{\ol{l}}=\sum_{\sigma \in
  \on{sh}(\ol{k},\ol{l})}\Intd{n+m}\om_{\sigma}.\eqf
\end{prop}
\begin{proof}Let $n = k_1 + ... + k_p$ and $m = l_1 + ... + l_q$. Then
  we have:
\begin{align*}\Intd{n}\om_{\ol{k}}\Intd{m}\om_{\ol{l}}& =\left(\Intd{n}
\fr{dt_1}{1-t_1}\cdots
\fr{dt_n}{t_n}\right)\left(\Intd{m}\fr{dt_{n+1}}{1-t_{n+1}}\cdots
\fr{dt_{n+m}}{t_{n+m}}\right)\\
&=\Intd{}\fr{dt_1}{1-t_1}\cdots \fr{dt_n}{t_n}
\fr{dt_{n+1}}{1-t_{n+1}}\cdots \fr{dt_{n+m}}{t_{n+m}}.
\end{align*}
The set $\Delta=\{0 <t_1<\ldots < t_n < 1\} \times \{0
<t_{n+1}<\ldots<t_{n+m} < 1\}$ can be, up to codimension $1$ sets, split
into a union of simplices
\[
\coprod_{\sigma \in
  \on{sh}(\ent{1,n},\ent{n+1,m})}\Delta_{\sigma}\qquad 
\mx{with }\Delta_{\sigma}=\{0 < t_{\sigma(1)}< t_{\sigma(2)} < ... <
  t_{\sigma(n+m)} < 1\},\]
where $\ent{a,b}$ denotes the ordered sequence of integers from $a$ to
$b$.

The integral over $\Delta$ is the sum of the integrals over the individual
simplices. But the integral over one of these simplices is, up
to the numbering of the variables, exactly one term of the sum $\ds
\sum_{\sigma\in\on{sh}(\ol{k},\ol{l})}\Intd{n+m} \om_{\sigma}$.
\end{proof}

\subsection{The stuffle relations in terms of integrals}\label{secstuint}
We explain here ideas 
already written in articles
of Goncharov \cite{PMMGon} and in Francis Brown's Ph.D. thesis
\cite{Brownphd}, showing how to express the stuffle relations
(\ref{serstu}) in terms of integrals.

\subsubsection{Example}We have $\zeta(2)=\Intd{2} \fr{dt_2}{t_2}
  \fr{dt_1}{1-t_1}$. The change of variables $t_2=x_1$ and
  $t_1=x_1x_2$ gives:
\[\zeta(2)=\Inc{2}\fr{dx_1}{x_1}\fr{x_1dx_2}{1-x_1x_2}=\Inc{2}\fr{dx_1
  dx_2}{1-x_1x_2}.\]
This change of variables is nothing but the blow-up of the point $(0,
0)$ in the projective plane, given in $n$ dimensions
by a sequence of blow-ups:
\eq \label{blowup} t_n = x_1,\; t_{n-1}=x_1x_2,\; \ldots,\; t_1 =
x_1...x_n.\eqf
We will write $\dx[n]$ for $dx_1\cdots dx_n$ where $n$ is the number
of variables under the integral. Using the change of variables
(\ref{blowup}) for $n = 4$ we write the Kontsevich forms as follows:
\[\zeta(4) =\Inc{4}\fr{\dx[4]}{1-x_1x_2x_3x_4}\,
,\quad \zeta(2,2)=\Inc{4}\fr{x_1x_2 \dx[4]}{(1-x_1x_2)(1-x_1x_2x_3x_4)}\]
and
\[
\zeta(2)\zeta(2)=\Inc{4}\fr{1}{(1-x_1x_2)}\fr{1}{(1-x_3x_4)}\dx[4]. 
\]
For any variables $\alpha$ and $\beta$ we have the equality:
\eq \label{formule1}\fr{1}{(1 -\alpha)(1 -\beta)}=\fr{\alpha}{(1 -
  \alpha)(1-\alpha \beta)}+\fr{\beta}{(1 -\beta)(1 - \beta \alpha )}
    +\fr{1}{1-\alpha \beta}.\eqf 
This identity will be the key of this section.

Setting $\alpha=x_1x_2$ and $\beta=x_3x_4$ and applying
(\ref{formule1}), we recover the stuffle relation: 
\begin{gather*}
\begin{split}
\zeta(2)\zeta(2)=\Inc{4}\left(\fr{x_1x_2}{
(1-x_1x_2)(1-x_1x_2x_3x_4) } \right.&
+\fr{x_3x_4}{(1-x_3x_4)(1-x_3x_4x_1x_2)}
\\
&\quad \left.+\fr{1}{1-x_1x_2x_3x_4}\right)\,
\dx[4]
\end{split}\\[2mm]
\zeta(2)\zeta(2)=\zeta(2, 2)+\zeta(2, 2)+\zeta(4).
\end{gather*}
\subsubsection{General case} We will show that the Cartier decomposition
(\ref{stupropeq}) below makes it possible to express all the stuffle
relations in terms of integrals as in the example above.

 Let $\mathbf{k}=(k_1,\ldots,k_p)$ and
$\mathbf{l}=(l_1,\ldots,l_q)$ be two tuples of integers with $k_1$, $l_1
\geqs 2$. As above, if $\sigma$ is a term of the formal sum
$\mathbf{k}*\mathbf{l}$, we will write $\sigma \in
\on{st}(\mathbf{k},\mathbf{l})$. We
will put an index $\sigma$ on any object which naturally depends on a
stuffle.

Let $\mathbf{k}=(k_1,\ldots,k_p)$ be as above and $n=k_1+\cdots+k_p$. We define
$\f{k}{p}$ to be the function of $n$ variables defined on $[0,1]^n$
given by:
\begin{multline*}\f{k}{p}(x_1,\ldots,x_n)=\fr{1}{1-x_1\cdots x_{k_1}} 
\fr{x_1\cdots  x_{k_1} }{1-x_1\cdots x_{k_1}x_{k_1+1}\cdots x_{k_1+k_2}} \\
  \fr{x_1\cdots x_{k_1+k_2}}{1-x_1\cdots x_{k_1+k_2+k_3}} \cdots
 \fr{x_1\cdots x_{k_1+...+k_{p-1}}}{1-x_1\cdots x_{k_1+\cdots +k_p}}.
\end{multline*}
\begin{prop}\label{intcuprop} For all $p$-tuples of integers
  $(k_1,\ldots,k_p)$ with $k_1\geqs 2$, we have (with $n = k_1 +\cdots
 + k_p)$:
\eq\label{intcu} \zeta(k_1,\ldots,k_p) =
\Inc{n}\f{k}{p}(x_1,\ldots,x_n)\dx[n].\eqf
\end{prop}
\begin{proof} Let $\om_{\mathbf{k}}$ be the Kontsevich form associated to
  a $p$-tuple $(k_1,\ldots, k_p)$ with $n=k_1 +\cdots + k_p$, so that
  $\zeta(k_1,\ldots, k_p)=\int_{\Delta_n}\oms k$.

Applying the change of variables (\ref{blowup}) to $\oms k$, we see that for each
term
$\fr{dt_i}{t_i}$, there arises from the $\fr{1}{t_i}$ a term
$\fr{1}{x_1 \cdots x_{n-i+1}}$ which cancels with
  $\fr{dt_{i-1}}{\cdots}=\fr{x_1\cdots
      x_{n-i+1}dx_{n-i+2}}{\cdots}$. This gives the result.
\end{proof}
To derive the stuffle relations in general using integrals and the
functions $\f{k}{p}$, we will use the following notation.

\begin{nota} Let $\mathbf{k}$ be a sequence $(k_1,\ldots, k_p)$, $n = k_1
  +\cdots + k_p$. We have $n$ variables $x_1,\ldots,x_n$.
\begin{itemize}
\item For any sequence $\mathbf{a}=(a_1,\ldots,a_r)$, we will write
  $\Prod \mathbf{a}=a_1\cdots a_r$. 
\item  The sequence $(x_1, \ldots, x_n)$ will be written
  $\mathbf{x}$. We set $\yks 1=(x_1,\ldots,x_{k_1})$ and
  \[\yks i =(x_{k_1+\cdots +k_{i-1}+1},\ldots , x_{k_1+\cdots +k_i}
  ),\]
so the $\mathbf{x}$ is the concatenation of sequences $\yks 1\cdots \yks p$. 
  
\item The sequence $(x_1,\ldots,x_{k_1+\cdots+k_i})=\yks 1 \cdots \yks
  i$ will be denoted
  by $\yks{\leqs i}$. If
  $\mathbf{k}=(\mathbf{k_0},k_p)$, $\ykz =\yks{\leqs p-1}$ will be the
  sequence 
\[
 (x_1,\ldots,x_{k_1+\cdots+k_{p-1}}).
\]
\item If $\mathbf{l}$ is a $q$-tuple with $l_1+ \cdots +l_q=m$ and
  $\sigma \in \on{st}(\mathbf{k},\mathbf{l})$, $y_{\sigma}$ will be the
  sequence in the variables $x_1,\ldots,x_n,x_1',\ldots, x_m'$ in
  which each group of variables
  \[\yks i=(x_{k_1+\cdots+k_{i-1}+1},\ldots , x_{k_1+\cdots+k_i}
  )\]
 \[ (\mx{resp. } \yls{j}=(x_{l_1+\cdots+ l_{j-1}+1}',
  \ldots,x_{l_1+\cdots+l_j}'))\] is in the
  position of $k_i$ (resp. $l_j$) in $\sigma$. Components of $\sigma$
  of the form $k_i +
  l_j$ give rise to subsequences like \[(x_{k_1+\cdots+k_{i-1}+1},\ldots ,
  x_{k_1+\cdots+k_i} , 
  x_{l_1+\cdots+l_{j-1}+1}', \ldots , x_{l_1+\cdots+l_j}' ) =( \yks{i}
  , \yls j ).\]

\item Following these notations, products $x_1 \cdots x_{k_1}$, 
  $x_{k_1+\cdots+k_{i-1}+1}\cdots x_{k_1+\cdots+k_i}$, $x_{1}\cdots
  x_{k_1+\cdots+ k_i}$ will be written
  respectively $\yk 1$, $\yk i$, $\ysk i$. As $\yks{\leqs p-1}=\ykz$
  and $\yks{\leqs p}=\mathbf{x}$,
   products $\ysk{p-1}$ and $\ysk p$ will be written $\Prod \ykz$ and
   $\Prod\mathbf{x}$.

We remark that for each $\sigma \in \on{st}(\mathbf k ,\mathbf l)$,
$\Prod \sigma= \Prod \mathbf x \Prod \mathbf{x'}$.
\end{itemize}
\end{nota}
\begin{rem}\label{remf}
Let $(k_1,\ldots, k_p)=(\mathbf{k_0},k_p)$ be a sequence of
  integers. Then: 
\[\f{k}{p}(\mb{x})
=\f{k}{p-1}(\yks{\leqs
p-1})\fr{\ysk{p-1}}{1-\ysk{p}}=\f{k}{p-1}(\ykz)\fr{\pykz}{1-\pyk}.\]
\end{rem}
\begin{prop}\label{stuprop}Let $\mathbf{k}=(k_1,\ldots , k_p)$ and
  $\mathbf{l}=(l_1,\ldots
  ,l_q)$ be two sequences of weight $n$ and $m$. Then:
\eq\label{stupropeq}\f{k}{p}(\yks{1},\ldots,\yks{p})\cdot
\f{l}{q}(\yls{1},\ldots ,\yls q ) =\sum_{\sigma \in
\on{st}(\mathbf{k},\mathbf{l})}
f_{\sigma}(y_{\sigma}).\eqf
\end{prop}
\begin{proof}We proceed by induction on the depth of the sequence. The
  recursion formula for the stuffle is given in (\ref{sturec}).

\textbf{If $\mathit{p = q = 1}$:} As we have
\begin{align*}f_n(\yks 1)f_m(\yls 1) & =\fr{1}{1 - \ysk 1}\cdot \fr{1}{1 -
  \ysl 1}=\fr{1}{1 - \pyk}\cdot \fr{1}{1 -
  \pyl}\, , \end{align*}
using the formula \eqref{formule1} with $\alpha=\pyk$ and $\beta=\pyl$ leads
to 
\begin{multline}
 f_n(\yks 1)f_m(\yls 1)  = \fr{\pyk}{(1 -\pyk)(1
-\pyk \pyl)}+
\fr{\pyl }{(1 -\pyl )(1 - \pyl \pyk )} \\
+ \fr{1}{1-\pyk \pyl }.
\end{multline}

\textbf{Inductive step:} Let $(k_1,\ldots , k_p) = (\mathbf{k_0}, k_p)$
and $(l_1,\ldots, l_q) =(\mathbf{l_0}, l_q)$ be two sequences. By Remark
\ref{remf}, the following equality holds
\[
f_{\mathbf{k_0},k_p}(\ykz , \yks p )f_{\mathbf{l_0},l_q}(\ylz , \yls q )=
f_{\mathbf{k_0}}(\ykz)\fr{\pykz}{1-\pyk }f_{\mathbf{l_0}}(\ylz)
\fr{\pylz}{1-\pyl}.
\]
Applying the formula \eqref{formule1} with $\alpha=\pyk$ and $\beta=\pyl$,
one sees that the RHS of the previous equation is equal to
\begin{multline*}f_{\mathbf{k_0}}(\ykz)f_{\mathbf{l_0}}(\ylz)\cdot
(\pykz \cdot \pylz) 
\left(\fr{\pyk}{(1-\pyk )(1-\pyk \pyl)}\right. \\[5mm]
\left.
+\fr{\pyl}{(1-\pyl)(1-\pyl \pyk)}
+\fr{1}{(1-\pyk \pyl)}\right).
\end{multline*}

Expanding and using Remark \ref{remf} we obtain:
\begin{multline} \label{stufunc}
f_{\mathbf{k_0},k_p}(\ykz ,\yks p )f_{\mathbf{l_0},l_q}(\ylz , \yls q)= \\[5mm]
\left(f_{\mathbf{k_0},k_p}({\mathbf{x}}) f_{\mathbf{l_0}}
    (\ylz)\right)\cdot\fr{\pyk \pylz}{1-\pyk \pyl} 
+ \left(f_{\mathbf{k_0}}(\ykz)f_{\mathbf{l_0},l_q}({\mathbf{x'}})\right)\cdot
  \fr{\pyl  \pykz}{1-\pyl  \pyk } \\[5mm]
 +\left(f_{\mathbf{k_0}}(\ykz
    )f_{\mathbf{l_0}}(\ylz)\right) \cdot \fr{\pykz \pylz}{1-\pyk 
    \pyl }.
\end{multline}

Hence, the product of functions $\f{k}{p}$ and $\f{l}{q}$ satisfies a recursion
formula identical to the formula (\ref{sturec}) defining the stuffle
product. Using
induction, the proposition follows.
\end{proof}
\begin{coro}[integral representation of the stuffle] Integrating the
  statement of the previous proposition over the cube and permuting
  the variables in each term of the RHS, we
obtain:
\[\zeta(\mathbf k)\zeta(\mathbf l)=\Inc{n}f_{\mathbf{k}}\dx[n]
\Inc{m}f_{\mathbf{l}}\dx[m] =\Inc{n+m}\sum_{\sigma
  \in \on{st}(\mathbf k,\mathbf l )}f_{\sigma}\; \dx[n+m] =\sum_{\sigma \in
  \on{st}(\mathbf{k},\mathbf{l})}\zeta(\sigma).\]
\end{coro}
\begin{proof} We only need to check that all integrals are convergent. As
  all the functions are positive on the integration domain, all
  variable changes are allowed, and we can deduce the convergence of
  each term from the convergence of the iterated integral representation for
  the multiple zeta values. 

Another argument is to remark that the orders of the poles of our
functions along a codimension $k$ subvariety can be at most $k$. Then, for
each integral, a succession of blow-ups ensures that the integral converges.
\end{proof}

\section{Moduli spaces of curves; double shuffle and forgetful maps}
\subsection{Shuffle and moduli spaces of curves}\label{section:shmod}
Let $\mathbf{k}$ and $\mathbf{l}$ be as in the previous section, let $n
=k_1+\cdots+k_p$ and $m =l_1+\cdots+l_q$. Following the article of
Goncharov and Manin \cite{MZMSGM}, we will identify a point of 
$\mn{j}$, the moduli space of curves of genus $0$ with $j+3$ marked points,
 with a sequence $(0, z_1,\ldots , z_j , 1,\infty)$, the $z_i$
being pairwise distinct and distinct from $0$, $1$ and $\infty$, and 
write
$\Phi_j$ for the open cell in $\mn{j}(\R)$ which is mapped onto
$\Delta_j$, the standard simplex, by the map: $\mn{j}\ra (\p^1)^j$,
$(0, z_1,\ldots , z_j , 1,\infty)\mapsto (z_1,\ldots , z_j)$. Then we 
have:
\[\zeta(k_1,\ldots , k_p) =\int_{\Phi_n}\om_{\mathbf{k}}.\]
\begin{prop}\label{oublishu}Let $\beta$ be the map defined by
\[\begin{array}{ccc}
\mn{n+m}&\xrightarrow[\hspace{.5cm}]{\beta}&\mn{n}\times
\mn{m} \\
(0, z_1,\ldots , z_{n+m}, 1,\infty)& \longmapsto{}{}&(0,z_1,\ldots , 
z_n,
1,\infty)\times(0, z_{n+1},\ldots , z_{n+m}, 1,\infty).
\end{array}\]
Then, letting $t_i$ be the coordinate such that $t_i(0,z_1,\ldots
,z_{n+m},1,\infty)= z_i$, we have
\[\beta^*(\om_{\mathbf{k}}\w \oms l)=\fr{dt_1}{1-t_1}\w \cdots\w
\fr{dt_n}{t_n }\w \fr{dt_{n+1}}{1-t_{n+1}}\w
\cdots\w\fr{dt_{n+m}}{t_{n+m}}.\] 
Furthermore, if for $\sigma \in \on{sh}(\ent{1, n},\ent{n + 1,
n + m})$ we write $\Phi_{n+m}^{\sigma}$ for the
open cell of $\mn{n+m}(\R)$ in which the points are
in the same order as their indices are in $\sigma$, we have
 \[\beta^{-1}(\Phi_ n \times \Phi_m)=\coprod_{\sigma \in
   \on{sh}(\ent{1, n},\ent{n + 1,
n + m})}\Phi_{n+m}^{\sigma}.\]
\end{prop}
\begin{proof}
The first part is obvious.

In order to show that  $\beta^{-1}( \Phi_n\times\Phi_m)=\coprod
\Phi_{n+m}^{\sigma}$ we have to remember that a cell in $\mn{n+m}(\R)$
is given by a cyclic order on
the marked points. Let $X =(0,z_1,\ldots , z_{n+m},1,\infty)$ be a
point in $\mn{n+m}(\R)$ such that $\beta(X)\in \Phi_n\times
\Phi_m$. The values of the $z_i$ have to be such that
\eq \label{conorsh}0<z_1<\ldots <z_n<1 \;(< \infty) \quad
\mx{and}\quad 0<z_{n+1}<\ldots
< z_{n+m}< 1 \;(< \infty).\eqf
However there is no order condition relating, say $z_1$ to $z_{n+1}$. 

So,
points on $\mn{n+m}(\R)$ which are in
$\beta^{-1}(\Phi_n\times \Phi_m)$ are such that the $z_i$ are
compatible with (\ref{conorsh}). That is, they are in $\ds \coprod_{\sigma \in
  \on{sh}(\ent{1, n},\ent{n + 1,
n + m})}\Phi_{n+m}^{\sigma}$.
\end{proof}
As  $\Phi_n \times \Phi_m \sm
\left(\beta(\beta^{-1}(\Phi_n \times \Phi_m)) \right)$ is of
codimension $1$, we have the following proposition.
\begin{prop}The shuffle relation $\zeta(\mb k)\zeta(\mb l)=\sum_{\sigma\in
\on{sh}(\mb k, \mb l)}\zeta(\sigma)$ is a consequence of the following
change of variables:
\[\int_{\Phi_n\times\Phi_m}\oms k\w \oms l
=\int_{\beta^{-1}(\Phi_n\times\Phi_m)}\beta^*(\oms k\w \oms l).\]
\end{prop}
\begin{proof}Using the previous proposition, the right hand side of
  this equality is equal to
\[\sum_{\sigma\in\on{sh}(\ent{1, n},\ent{n + 1,
n + m})}\int_{\Phi_{n+m}^{\sigma}}
\fr{dt_1}{1 -t_1}\w \cdots\w \fr{dt_{n+m}}{t_{n+m}}.\]
Then we permute the variables and change their names in order to have
an integral over  $\Phi_{n+m}$ for each term. This is the same
computation we did for the integral over $\R^{n+m}$ in proposition
\ref{shintprop}.

As the form $\fr{dt_1}{1 -t_1}\w \cdots\w \fr{dt_{n+m}}{t_{n+m}}$
(resp. $\fr{dt_{\sigma(1)}}{1 -t_{\sigma(1)}}\w \cdots\w
    \fr{dt_{\sigma(n+m)}}{t_{\sigma(n+m)}}$) does not have any poles on
    the boundary of $\Phi_{n+m}^{\sigma}$ (resp. $\Phi_{n+m}$), all the
    integrals are convergent.
\end{proof}
\subsection{Stuffle and moduli spaces of curves}\label{modstu}
In Section \ref{secstuint},  in order to have an integral
representation of the stuffle product, we recalled, using the integral
over a simplex and a change of variables, a cubical representation  of
the MZVs  (integral over a cube). We use here a similar change of
variables to introduce another system of local coordinates on
$\mnb{r}$, the Deligne-Mumford compactification of the moduli space of
curves. Following \cite{Brownphd}, we will speak of cubical coordinates. Those cubical coordinates, $u_i$,  are defined on an open subset of $\mnb{r}$ 
by $u_1 = t_r$ and $u_i =
t_{r-i+1}/t_{r-i+2}$ for $i < r$ where the $t_i$ are the usual (simplicial)
coordinates on $\mnb{r}$. This cubical system is well adapted
to express the stuffle relations on the moduli spaces of curves.
\begin{prop}\label{stmod}Let $\delta$ be the map defined by
\begin{align*}\mn{n+m}&\xrightarrow[\hspace{1cm}]{\delta}
  \mn{n}\times \mn{m} \\
(0,z_1,\ldots ,z_{n+m}, 1,\infty)&\longmapsto  (0, z_{m+1},\ldots ,
z_{m+n}, 1,\infty)\times(0,z_1,\ldots , z_m,
z_{m+1},\infty).
\end{align*}

Writing the expression of $\oms k$ and $\oms l$ in cubical coordinates,
one finds $\oms k =f_{\mathbf{k}}(u_1,\ldots ,u_n)\du[n]$ and $\oms l =
f_{\mathbf{l}}(u_{n+1},\ldots ,u_{n+m})\du[m]$ where the $f_{\mathbf{k}}$ are as
in section \ref{secstuint}. Then, using those coordinates we have
\[\delta^*(\oms k\w \oms l)=\f{k}{p} (u_1,\ldots
,u_n)\f{l}{q}(u_{n+1},\ldots , u_{n+m})\du[n+m]\] 
and
\[\delta^{-1}(\Phi_n\times \Phi_m) =\Phi_{n+m}.\]
\end{prop}
\begin{proof}To prove the second statement, let $X = (0, z_1,\ldots ,
  z_{n+m},1,\infty)$ be such that $\delta(X)\in \Phi_ n\times
  \Phi_m$. Then the values of the $z_i$'s have to verify
\eq 0 < z_1 <\ldots < z_m < z_{m+1}\;(< \infty)\quad \mx{and} \quad 0 <
z_{m+1} <\ldots < z_{n+m} < 1\;(< \infty).\label{conorst}\eqf
These conditions show that $0 < z_1 <\ldots < z_m < z_{m+1} <\ldots <
1 < \infty$, so $X\in\Phi_{n+m}$.

To prove the first statement, we claim that $\delta$ is expressed in
cubical coordinates by
\[(u_1,\ldots , u_{n+m})\longmapsto (u_1,\ldots , u_n) \times
(u_{n+1},\ldots , u_{n+m}).\]
It is obvious to see that for the left hand factor the coordinates are
unchanged. For the right hand factor we have to rewrite the
expression of the right side in terms of the standard representatives
 on $\mn{m}$. We have
\[(0, z_1,\ldots , z_m, z_{m+1},\infty)= (0,z_1/z_{m+1},\ldots ,
z_m/z_{m+1},1,\infty)=(0,t_1,\ldots,t_m,1,\infty)\]
in simplicial coordinates. This point is given in cubical coordinates
on $\mn{m}$ by
\[(t_m, t_{m-1}/t_{m},\ldots , t_1/t_2)= (z_m/z_{m+1},\ldots , z_1/z_2)
=(u_{n+1},\ldots,u_{n+m}).\]
\end{proof}
As a consequence of this discussion and the results of Section
\ref{secstuint}, we have the following proposition.
\begin{prop}\label{oublistu2}Using the Cartier decomposition
  (\ref{stupropeq}), the
  stuffle product can be viewed as the change of variables:
\[\int_{\Phi_n\times \Phi_m}\oms k\w \oms l
=\int_{\delta^{-1}(\Phi_n\times \Phi_ m)}\delta^*(\oms k\w \oms l).\]
\end{prop}
 \begin{rem}\label{remstupb}We should point out here the fact
   that the Cartier decomposition "does not lie in the moduli
 spaces of curves", in the sense that forms appear in the decomposition
 which are not holomorphic on the moduli space. For example, in the
 Cartier decomposition of $f_{2,1}(u_1, u_2, u_3)f_{2,1}(u_4,
 u_5,u_6)$, we see the term
 \[\fr{u_1u_2u_4u_5du_1du_2du_3du_4du_5du_6}{
 (1-u_1u_2u_4u_5)(1-u_1u_2u_3u_4u_5u_6)}\]
 which is not a holomorphic differential form on $\m_{0,6}$. However,
 it is a well-defined convergent form on the standard cell where it is
 integrated. Changing the numbering of the variables (which stabilises
 the standard cell) gives the equality with $\zeta(4, 2)$. This example
represents the situation in the general case: when
 simply dealing with integrals, the non-holomorphic forms
 are not a problem. However, in the context of framed motives, they are.
 \end{rem}
\section{Motivic shuffle for the "convergent" words}
\subsection{Framed mixed Tate motives and motivic multiple zeta values}
This section is a short introduction to the motivic tools we will use
to prove the motivic double shuffle. The motivic context is a cohomological
version
of Voevodsky's category $\mc{DM}_{\Q}$ \cite{Vo00}. Goncharov developed  
in \cite{VHMGon}, \cite{GSFGGon} and \cite{MPMTMGon}
an additional structure on mixed Tate motives, introduced in \cite{BGSV}, in
order to select
a specific period of a mixed Tate motive.

An $n$-framed mixed Tate motive is a mixed Tate motive $M$ equipped
with two non-zero morphisms:
\[v:\Q(-n)\ra \Gr^W_{2n} M \qquad f:\Q(0)\ra \left(\Gr^W_{0} M
\right)^{\vee}=\Gr^W_{0}M^{\vee}.\]
On the set of all $n$-framed mixed Tate motives, we consider the coarsest
equivalence relation under which $(M,v,f)\sim (M',v',f')$ if there is a
linear map $M \ra M'$ respecting the frames. Let $\mc{A}_n$ be the set 
of
equivalence classes and $\mc{A}_{\bullet}$ the direct sum of the
$\mc{A}_n$. We write $[M;v;f]$ for an equivalence class. 
\begin{thm}[{\cite{GSFGGon}}]\label{FMTMhopf}$\mc{A}_{\bullet}$ has a
  natural structure of graded commutative Hopf
  algebra over $\Q$.

$\mc{A}_{\bullet}$ is canonically isomorphic to the dual of the Hopf 
algebra of
all endomorphisms of the fibre functor of the Tannakian category of 
mixed
Tate motives.
\end{thm}
In our context, the morphism $v$ of a
frame should be linked with some differential form 
and the morphism $f$ is a homological counterpart of $v$, that is a
real simplex.

We give here two technical lemmas that will be used in the
next
sections. We write $[M,v,f]$ for the equivalence class of $(M,v,f)$ in
$\mc{A}_{\bullet}$. By a slight abuse of notation, we will speak of
framed mixed Tate motives refering to both $(M,v,f)$ and $[M,v,f]$.

We recall that the addition of two framed mixed Tate motives $[M, v,f]$ 
and
$[M',v',f']$ is given by
\[
 [M, v,f]\oplus[M',v',f']:\!= [M\oplus M',(v,v'), f+f'].
\]

\begin{lem}\label{som}Let $M$ be a mixed Tate motive. $v,\,v_1,\,v_2
 :\Q(-n)\ra\on{Gr}^W_{2n}M$ and $f,\,f_1,\, f_2: \Q(0) \ra
  \on{Gr}^W_{0} M^{\vee}$. We have:
\[[M; v; f_1 + f_2] = [M; v; f_1] + [M; v; f_2] \]
and 
\[
[M; v_1+v_2; f] = [M; v_1; f] + [M; v_2; f]. 
\]
\end{lem}
\begin{proof}It follows directly from the definition in \cite{GSFGGon}. 
For the first case, it is straightforward to check that the diagonal map
  $\vp: M\ra M\oplus M$ is compatible with
  the frames.
For the second equality, the map from $M\oplus M$ to $M$ which sends
$(m_1,m_2)$ to $m_1+m_2$ gives the map between the underlying vector spaces and
respects the frames. 
\end{proof}
\begin{lem}\label{idio}Let $M$ and $M'$ be two mixed Tate motives. Let
  $M$ be framed by $v: \Q(-n)\ra \on{Gr}^W_{2n}$ and $f:\Q(0)\ra
  \on{Gr}^W_{0}M^{\vee}$. Suppose there exists $v':\Q(-n)\ra
  \on{Gr}^W_{2n}M'$ and 
  $\vp:M'\ra M$ compatible with $v$ and $v'$. Then $f$ induces a map
  $f':\Q(0)\ra 
  \on{Gr}^W_{0}M'^{\vee}$ and if $f'$ is non zero, then $\vp$ gives an equality of framed
  mixed 
  Tate motives $[M;v;f]=[M;v';f']$.
\end{lem}

We recall a classical result, used in \cite{MZMSGM} and
described more explicitly in
\cite{PMMGon}, that allows us to build mixed Tate motives from natural
geometric situations.
 In \cite{PMMGon}, A.B. Goncharov defined a \emph{Tate variety} as a
 smooth projective variety $\mc M$ such that the
 motive of $\mc M$ is a direct sum of copies of the Tate motive $\Q(m)$ (for
 certain $m$). We say
 that a divisor $D$ on $\mc M$ provides a Tate stratification on $\mc
 M$ if all strata of $D$, including $D_{\emptyset}=\mc M$, are Tate varieties. 

Let $\mc M$ be a smooth  variety and $X$ and $Y$ be two normal
crossing divisors on $\mc M$. Let $Y^X$ denote $Y\sm (Y\cap X$), which is a
normal crossing divisor on $\mc M\sm X$.
 
\begin{lem}\label{lemexismtm}Let $\mc M$ be a smooth variety of dimension $n$ over $\Q$ and
$X\cup Y$ be a 
  normal crossing divisor on $\mc M$ providing a Tate stratification
  of $\mc M$. If $X$ and $Y$ share no common irreducible
  components then 
  there exists a mixed Tate motive:
\[\HH^n(\mc M\sm X;Y^X)\]
such that its different realisations are given by the respective
relative cohomology groups.
\end{lem}

We have the following version given in \cite{MZMSGM}.
\begin{coro}\label{motmzv} Let $X$ and $Y$ be two normal crossing
  divisors on 
  $\partial \mnb{n}$ and suppose they do not share any irreducible
  components.  
Then, any choice of non-zero elements
\[[\om_X]\in\Gr^W_{2n}(\HH^n(\mnb{n} \sm X));\qquad [\Phi_Y] \in
\Gr^W_{0}(\HH^n(\mnb{n}; Y))^{\vee}\]
defines a \emph{framed} mixed Tate motive given by
\[
\left[\HH^n(\mnb{n}\sm X;Y^X);[\om_X];[\Phi_Y]\right].
\]
\end{coro}
The following lemma shows that we have some
flexibility in choosing $X$ and $Y$ for the framed mixed Tate motive
$\left[\HH^n(\mc M\sm
X;Y^X);[\om_X];[\Phi_Y]\right]$.
\begin{lem}\label{inclmot}With the notation of Lemma \ref{lemexismtm}, 
let  $X'$ be a normal
  crossing divisor containing $X$ which still does not share any irreducible
  component with $Y$, $X'\cup Y$ being a normal crossing divisor. Then:
 \[\left[\HH^n(\mc M\sm X;Y^X); [\om_X];[\Phi_Y]\right]
 = \left[ \HH^n(\mc M\sm X';Y^{X'}); [\om_X];[\Phi_Y]
 \right].\] 
Suppose now that $Y'$ is a normal crossing divisor containing $Y$
which does not share any irreducible component with $X'$, $X'\cup Y'$ being a normal crossing divisor. Then:
\[ \left[ \HH^n(\mc M\sm X';Y^{X'}); [\om_X];[\Phi_Y]
 \right]=\left[\HH^n(\mc M\sm X';Y'^{X'}); [\om_X];[\Phi_Y]\right]. 
\]
\end{lem}
We are now in a position to introduce Goncharov and Manin's definition
of motivic multiple zeta values.
\begin{defn}
In particular, let $\mathbf{k}$ be a $p$-tuple with $k_1\geqs 2$ and let $\As k$
be the
divisor of singularities of $\oms k$. Let $B_n$ be the Zariski closure
of the boundary of $\Phi_n$. The motivic multiple zeta value is defined in
\cite{MZMSGM}
by:
\[\left[\HH^n(\mnb{n}\sm \As k;\abs{n}{k});[\oms k];[\Phi_n]\right].\]  
\end{defn}

\subsection{Motivic Shuffle} \label{sec:motshu}
The map $\beta$ defined in Proposition \ref{oublishu}
will be the key to check that the motivic multiple zeta values satisfy the
shuffle
relations. This map extends continuously to the Deligne-Mumford
compactification of the moduli spaces of curves:
\[\begin{array}{ccc}
\mnb{n+m}&\xrightarrow[\hspace{1cm}]{\beta}&\mnb{n}\times
\mnb{m}.
\end{array}\]
 Let $\om_{\mathbf{k}}$ 
and $\om_{\mathbf{l}}$ be as in section \ref{section:shmod}, and write $\As k$
and $\As l$ for
their respective divisors of singularities. Let $B_n$
and $B_m$ denote the Zariski closures of the boundary of $\Phi_n$ and
$\Phi_m$ respectively. For $\sigma \in
\on{sh}(\ent{1,n},\ent{n+1,n+m})$, let $\om_{\sigma}$ denote the
differential form which corresponds
to the shuffled MZV and let $A_{\sigma}$ denote its divisor of
singularities. Let $B_{n+m}$ denote the Zariski
closure of the boundary of 
$\Phi_{n+m}$ and $B_{\sigma}$ that of $\Phi_{n+m}^{\sigma}$.
The shuffle relations between motivic multiple zeta values are given in
the following proposition.
\begin{prop}\label{motivicshu}We have an equality of framed motives:
\begin{multline*}
\left[\HH^n\left(\mnb{n}\sm \As k;\abs{n}{k}\right); [\oms k];
  [\Phi_n]\right]\cdot
    \left[\HH^m \pa{\mnb{m}\sm \As l; \abs{m}{l}}; [\oms l];
      [\Phi_m]\right]=\qquad\\[3mm]
 \sum_{\sigma \in
    \on{sh}(\ent{1,n},\ent{n+1,n+m})}
  \left[\HH^{n+m}\pa{\mnb{n+m}\sm A_{\sigma};B_{n+m}^{A_{\sigma}}};
    [\om_{\sigma}];[\Phi_{n+m}]\right]. 
\end{multline*}
\end{prop}
\begin{proof}To prove this equality,
  we need to display a map between the underlying vector spaces which
  respects the frames.

Let $A'$  be the boundary of $(\mnb{n}\sm \As
k)\times(\mnb{m}\sm \As l)$, it is equal to the divisor of
singularities of $\oms k \w \oms l$ on $\mnb{n}\times\mnb{m}$.

 Let $\ds A_0=\beta^{-1}(A')$ and let $B_0$ be the Zariski closure of the
boundary of
  $\Phi_0=\beta^{-1}(\Phi_n\times \Phi_m)$. Let $B_{n,m}$ be the Zariski
  closure of the boundary of $\Phi_n\times\Phi_ m$. The map $\beta$
  induces a map: 
\[\xymatrix{(\mnb{n+m}\sm A_0;B_0^{A_0})\ar[r]^<{\beta} &
\left((\mnb{n}\sm \As
k)\times (\mnb{m}\sm \As l);\beta(B_0)^{A'}\right)
\\
& \ar@{^(->}[u]^{\alpha}
\left((\mnb{n}\sm \As k)\times (\mnb{m} \sm \As l);B_{n,m}^{A'}\right).}\]

We introduce 
the vertical inclusion $\alpha$ because $B_0$ does not map onto
$B_{n,m}$ via $\beta$. The map $\alpha$ induces a map on the mixed Tate motives:
\eq \label{shm0}
\begin{array}{l}\ds
\HH^{n+m}\pa{(\mnb{n}\sm \As k)\times(\mnb{m}\sm \As
  l);\beta(B_0)^{A'}} 
\xrightarrow[\hspace{1pc}]{\alpha^*} \hspace{3.5cm}~ \\[2mm]
\multicolumn{1}{r}{
\ds \HH^{n+m} \pa{(\mnb{n}\sm \As k)\times
    (\mnb{m}\sm \As l);B_{n,m}^{A'}}.
}\end{array}
\eqf

The frames on the RHS of (\ref{shm0}) are given by $[\Phi_n\times
\Phi_m]$ and $[\oms k\w \oms l]$.
Applying lemma \ref{idio} to (\ref{shm0}), $[\Phi_n \times \Phi_m]$ 
induces a map $\tilde{\Phi}$ from $\Q(0)$ to the $-2(n+m)$ graded part of the
LHS of (\ref{shm0}). In fact, since $\alpha$ is the identity map, we
have $[\tilde{\Phi}]=[\Phi_n\times\Phi_m]$, so $[\Phi_n\times \Phi_m]$
and $[\oms k \w \oms l]$ give frames
on the LHS of (\ref{shm0}) which are compatible with the map $\alpha^*$.

The map $\beta$ induces a map on the mixed Tate motives:
\begin{multline}\label{shm1}
\HH^{n+m}\pa{(\mnb{n}\sm \As k)\times(\mnb{m}\sm \As
  l);\beta(B_0)^{ A'}}
\xrightarrow[\hspace{1pc}]{\beta^*} \\ 
  \HH^{n+m}(\mnb{n+m}\sm A_0;\ab{0}{0}).
\end{multline}
On the RHS of (\ref{shm1}) the frames are
given by $[\om_0]$ where $\om_0$ is $\beta^*(\oms k\w \oms l)$ and
$[\Phi_0]=[\beta^{-1}(\Phi_n\times
\Phi_m)]$, which are compatible with the map $\beta^*$.

Now we can prove the proposition.
The K\"unneth formula gives a map:
\begin{multline*}
\HH^n\pa{\mnb{n}\sm \As k;\abs{n}{k}}\otimes \HH^m\pa{\mnb{m}\sm
  \As l;\abs{m}{l}}\xrightarrow{\hspace{2pc}}
\\[2mm]
\HH^{n+m}\pa{(\mnb{n}\sm \As k)\times(\mnb{m}\sm \As l);
  B_{n,m}^{A'}
}.
\end{multline*}
By theorem \ref{FMTMhopf}, this map also respects the frames, so the
associated framed mixed Tate motives are equal. By (\ref{shm0}), 
\[
\left[
\HH^{n+m}\pa{(\mnb{n}\sm \As k)\times(\mnb{m}\sm \As l);
  B_{n,m}^{A'}};[\oms k \ot \oms l];[\Phi_n\times \Phi_m] 
\right] 
\]
is equal to 
\[
\left[\HH^{n+m}\pa{(\mnb{n}\sm \As k)\times(\mnb{m}\sm \As
  l);\beta(B_0)^{A'}};[\oms k \ot \oms l]; [\Phi_n\times \Phi_m]
\right] ,
\]
which, using \eqref{shm1}, is equal to 
\[
\left[
 \HH^{n+m}(\mnb{n+m}\sm A_0;\ab{0}{0}) ; [\om_0] ; [\Phi_0]
\right].
\]


It remains to show that
\begin{multline}
 \label{shm2} \left[\HH^{n+m}(\mnb{n+m}\sm
  A_0;\ab{0}{0});[\om_0];[\Phi_0]\right]=  \\ 
\sum_{\sigma} \left[
  \HH^{n+m} (\mnb{n+m}\sm A_{\sigma};\ab{n+m}{\sigma});
  [\om_{\sigma}]; [\Phi_{n+m}]\right].
\end{multline}


In the LHS of (\ref{shm2}), $B_0$ being included in
$B_{\on{sh}}=\bigcup_{\sigma}B_{\sigma}$,
we can replace $B_0$ by $B_{\on{sh}}$ using lemma \ref{inclmot}.

As $[\Phi_0]=\sum_{\sigma}[\Phi_{n+m}^{\sigma}]$, lemma \ref{som} shows that
the LHS of \ref{shm2} is equal to
\[\sum_{\sigma}\left[\HH^{n+m}(\mnb{n+m} \sm
  A_0;\ab{\on{sh}}{0});[\om_0];[\Phi_{n+m}^{\sigma}]\right]. \]
Using the fact that $B_{\sigma}\subset B_{\on{sh}}$ and an inclusion map,
lemma \ref{inclmot} shows that this framed motive is equal to
\[
\sum_{\sigma}\left[\HH^{n+m}(\mnb{n+m}\sm A_0;\ab{\sigma}{0});[\om_0];
[\Phi_{n+m}^{\sigma}]\right]. 
\]
As the divisor of singularities $A$ of $\om_0$ is included in $A_0$,
using lemma \ref{inclmot} we can replace $A_0$ by $A$ in this
framed motive. Then permuting the points gives an equality of framed
motives on each term of the sum, 
\[
\left[\HH^{n+m}(\mnb{n+m}\sm
  A_0;\ab{\sigma}{0});[\om_0];[\Phi_{n+m}^{\sigma}]\right],
  \]
  with
\[ 
\left[
  \HH^{n+m}(\mnb{n+m}\sm A_{\sigma};\ab{n+m}{\sigma});[\om_{\sigma}];
  [\Phi_{n+m}] \right].
  \]
Thus, we obtain the desired formula:
\[\begin{array}{l}\ds
\left[\HH^n\pa{\mnb{n}\sm \As k;B_n^{\As k}}; [\oms k];
  [\Phi_n]\right] \cdot \left[\HH^m \pa{\mnb{m}\sm \As l;B_m^{\As l}}; 
[\oms l]; [\Phi_m]\right]=\hspace{.5cm}~\\[2mm]
\multicolumn{1}{r}{\ds
\sum_{\sigma\in\on{sh}((1,\ldots,n),(n+1,\ldots,n+m))} \left[\HH^{n+m}
  \pa{\mnb{n+m}\sm A_{\sigma};B_{n+m}^{A_{\sigma}}};
  [\om_{\sigma}]; [\Phi_{n+m}]\right].}\end{array}\]

\end{proof}

\section{The stuffle case}
The goal of this section  is to be able to translate all the
calculations done in Section \ref{secstuint} into a motivic context. 
In order to achieve this goal, we need to define, for all $n$
greater than $2$, a variety $X_n\ra \A^n$ resulting from successive blow-ups of
$\A^n$ together with a differential form $\Om_{k_1,\ldots ,k_p}^s$ for
any tuple of integer $(k_1,\ldots, k_p)$ (with $k_1+\cdots k_p=n$) and any
permutation $s$ of $\ent{1,n}$. We use the $X_n$ to give another
definition of the motivic multiple zeta value which we show is
actually equal to Goncharov-Manin's. Then, using a natural map
from $X_{n+m}$ to an open subset of  $\m_{0,n+m+3}$, we use this new
definition to prove  that the motivic multiple zeta values satisfy the
stuffle relation. 

\subsection{Blow-up preliminaries}

\begin{lem}[Flag Blow-up Lemma; {
\cite{PCCSUly}}.]
\label{flbllem} Let $V^1_0 \subset V^2_0 \subset \cdots V^r_0 \subset W_0$ be
a flag of smooth subvarieties in a smooth algebraic variety $W_0$. For $k =
1, \ldots, r$,
define inductively $W_k$ as the blow-up of $W_{k-1}$ along $V^k_{k-1}$, then 
$V^k_k$ as the exceptional divisor in $W_k$ and $V^i_k$, $k\leqs i$, as the
proper transform of $V^i_{k-1}$ in $W_k$. Then the
preimage of $V^r_0$ in the resulting variety $W_r$ is a normal crossing divisor
$V^1_r \cup \cdots \cup V^r_r$.
\end{lem}

If $\mathscr F$ is a flag of subvarieties $V^i_0$ in a smooth algebraic variety
$W_0$ as in the previous lemma, the resulting space $W_s$ will be denoted by
$\on{Bl}_{\mathscr F} W_0$.

\begin{thm}[{\cite{COVHu}}]\label{blsq}Let $X_0$ be an open subset of a
nonsingular algebraic variety $X$.
Assume that $X \sm X_0$ can be decomposed as a finite union $\cup_{i \in I}
D_i$ of closed irreducible 
subvarieties such that
\begin{enumerate}
 \item for all $i\in I$, $D_i$ is smooth;
\item for all $i,j \in I $, $D_i$ and $D_j$ meet cleanly, that is
the scheme-theoretic intersection is smooth and the intersection of the
tangent space $T_X(D_i)\cap T_X(D_j)$ is the tangent space of the
intersection $T_X(D_i \cap D_j)$;
\item for all $i,j \in I $, $D_i\cap D_j = \emptyset$ ; or a disjoint union of
$D_l$.
\end{enumerate}
The set $\mc D = \{D_i\}_{i \in I}$ is then a poset. Let $k$ be the rank of
$\mc D$. Then there is a sequence 
of well-defined blow-ups
\[\on{Bl}_{\mc D} X \ra \on{Bl}_{\mc D \leqs k-1} X \ra \cdots \ra \on{Bl}_{\mc
D \leqs 0}
X\ra X\]
where $\on{Bl}_{\mc D \leqs 0} X\ra X$ is the blow-up of $X$ along $D_i$ of rank
$0$,
and, inductively, 
$\on{Bl}_{\mc D \leqs r} X \ra \on{Bl}_{\mc D \leqs r-1} X$ is the blow-up of
$\on{Bl}_{\mc D 
\leqs r-1} X$ along the proper transforms of
$D_j$ of rank $r$, such that
\begin{enumerate}
 \item $\on{Bl}_{\mc D} X$ is smooth;
\item $\on{Bl}_{\mc D} X \sm X_0 = \bigcup_{i \in I}\wt{D_i}$ is a divisor with
normal crossings;
\item for any integer $k$, $\wt{D_{i_1}}\cap \cdots \cap \wt{D_{i_k}}$ is
non-empty if and only if, up to numbering, $D_{i_1} \subset \cdots \subset D_{i_k}$ form a chain
in the
poset $\mc D$. Consequently, $\wt{D_i}$ and $\wt{D_j}$ meet if and only if
$D_i$ and $D_j$ are
comparable.
\end{enumerate}
\end{thm}

The fact that blow-ups are local constructions yields directly  the following corollary.

\begin{coro}[Flags blow-up sequence]\label{flblsq}
 Let $X$ and $ \mc D$ be as in the previous theorem. Let $\mscr F_1, \ldots ,
\mscr F_k$ be flags of subvarieties of $\mc D$ such that
\begin{enumerate}
 \item $\mscr F_1, \ldots ,\mscr F_k$ is a partition of $\mc D$,
\item if $D$ is in some $\mscr F_i$, then for all $D' \in \mc D$ with $D' <D$
there exists some $j\leqs i $ such that $D' \in \mscr F_j$.
\end{enumerate}
If $\mscr F^i_j$ denotes the flag of the proper transform of elements of
$\mscr F^{i-1}_j$ in 
\[\on{Bl}_{\mscr F^{i-1}_i}\left(\cdots
\left(\on{Bl}_{\mscr F_1}X
\right) \cdots \right),\] 
then
\[ \on{Bl}_{\mc D} X = \on{Bl}_{\mscr F^{k-1}_k}\left(\cdots
\left(\on{Bl}_{\mscr F_1}X
\right) \cdots \right).
\] 
We will denote such a sequence of blow-ups by 
\[
 \on{Bl}_{\mc F_k, \ldots, \mc F_1} X.
\]

\end{coro}

As we want to apply these results in order to have a motivic description of the
stuffle product in terms of blow-ups, we need some more precise
information about what sort 
of motives arise from the construction of Theorem \ref{blsq}. Following the
notation of the article \cite{COVHu}, in particular using the proof of theorems
1.4, 1.7 and Corollary 1.6, we deduce the following proposition:

\begin{prop}\label{Tatebl}
Suppose that $X$ and $\mc D=\cup D_i$ as in proposition \ref{blsq} are such
that $X$ and all the $D_i$ are Tate varieties. Let $\mc E^{r+1}$ be the set of
exceptional divisors of $\on{Bl}_{\mc D^{\leqs r}}X\ra X$. Then all possible
intersections of strata 
of $\mc D^{ r+1}\cup \mc E^{r+1}$ are Tate Varieties and so is   $\on{Bl}_{\mc
D^{\leqs r}}X$.
\end{prop}
\begin{proof} Mainly following the proof of theorem 1.7 in \cite{COVHu}, we use
an induction on $r$.

If $r=0$ then $\on{Bl}_{\mc D^{\leqs 0}}X\ra X$ is the blow-up along the disjoint
subvarieties $D_i$ of rank $0$.

All the exceptional divisors in $\mc E^1$ are of the form $\p (N_X D_i) $ (with
$D_i$ of rank $0$), and as the $D_i$ are Tate, so are the exceptional divisors.

The blow-up formula
\begin{align}
h(X_Z)=\h(X) \bigoplus_{i=0}^{d-1} h(Z)(-i)[-2i]
\end{align}
tell us that the blow-up of a Tate variety  $X$ along some Tate variety 
$Z$ of codimension $d$ is
a Tate variety. Then  $\on{Bl}_{\mc D^{\leqs 0}}X$ is Tate. Moreover, 
if
$D_i^1$ is an element of $\mc D^1$, then it is the proper transform of an 
element
$D_i$ in $\mc D$ of rank greater than $1$. Theorem 1.4 in 
\cite{COVHu} tells
us that $D^1_i=\on{Bl}_{D_j \subset D_i ; rank(D_j)=0}D_j$ and is
therefore  a Tate variety.

We now need to show that all intersections of strata of $\mc D^1\cap \mc E^1 $
are Tate. As $\on{Bl}_{\mc D^{\leqs 0}}X\ra X$ is a blow-up along
disjoint subvarieties, the exceptional divisors are disjoint and we
conclude that elements in $\mc E^1$ do not intersect. 

Let $D^1_i$ and $D^1_j$ be two elements of $\mc D^1$ ; they are the 
proper transforms
of $D_i$ and $D_j$ in $\mc D$. If $D_i \cap D_j=\emptyset$ then, the same
hold for their proper transforms and there is nothing to
prove. Otherwise, by assumption, $D_i \cap 
D_j$ is  a a disjoint union $\cup D_l$. If the maximal rank of the
$D_l$ is $0$ then Lemma 2.1 in 
\cite{COVHu} ensures that the proper transforms have an empty intersection.   If
the maximal rank of the $D_l$ is greater than $1$ the fact that $D_i$ and $D_j$
meet cleanly ensures that the proper transform of the intersection is the
intersection of the proper transforms, that is 

\[D_i^1 \cap D_j^1= \on{Bl}_{D_l \subset D_i\cap D_j ; rank(D_l)=0}D_i\cap D_j
\]
and the intersection is Tate because $D_i\cap D_j$ is a disjoint union 
of $D_l$ which are Tate.
Moreover from theorem 1.4 (\cite{COVHu}) we have  $D_i^1 \cap D_j^1= \cup
D_l^1$. Thus we can consider only intersections of the form $E^1 \cap D_i^1$
with $E^1$ in $\mc E^1$ and $D_i^1$ in $\mc D^1$. Such an intersection is non
empty if and only if $E^1$ comes from an  element $D_j$ of rank $0$ in $\mc D$ with $D_j \subset D_i$.
Then $E^1 \cap D_i^1$ is $\p(N_{D_i}D_j)$ and is a Tate variety.

\textbf{Assume the statement is true for $\on{Bl}_{\mc D^{\leqs r-1}}X$, $\mc
E^r$ and $\mc D^r$.} By corollary 1.6 in \cite{COVHu}, the blow-up
$\on{Bl}_{\mc D^{\leqs r}}X\ra \on{Bl}_{\mc D^{\leqs r-1}}X$ is 

\[\on{Bl}_{\mc D_{\leqs 0}^r}\left(\on{Bl}_{\mc D^{\leqs r}}X\right)\lra
\on{Bl}_{\mc D^{\leqs r-1}}X.
\]

This is a blow-up along elements in $\mc D^r$ of rank $r$ which by
assumption are Tate, as $\on{Bl}_{\mc D^{\leqs r-1}}X$. Then, $\on{Bl}_{\mc
D^{\leqs r}}X$ and the new exceptional divisors are Tate. The other exceptional
divisors are proper transforms of elements in $\mc E^r$ and are of the form 

\[E_i^{r+1} =\on{Bl}_{E_i^r \cap D_l^r ; rank(D_l)=r}E_i^r
\]
with $E_i^r$ in $\mc E^r$ and $D_l^r$ in $\mc D^r$ coming from some $D_l$ in
$\mc D$. As by the induction hypothesis both $E_i^r$ and $E_i^r\cap D_l^r$ are
Tate, $E_i^{r+1}$ is a Tate variety. The same argument proves that all elements
in $\mc D^{r+1}$ are Tate. As previously the intersection of two elements in
$\mc D^{r+1}$ is either empty or the proper transform of the intersection of
two elements in $\mc D^r$, again this proper transform is Tate.

Theorem 1.4 tells us that  the intersection $D_i^{r+1}\cap D_j^{r+1}$ of two
elements of $\mc D^{r+1}$ is either empty or the union of some elements
$D_l^{r+1}$ in $\mc D^{r+1}$. Then, to  prove that all possible intersections
of strata of $\mc E^{r+1}\cup\mc D^{r+1}$ is Tate, it is enough to prove that
the intersection of some $D_i^{r+1}$ with any intersection $E_1^{r+1}\cap
\cdots E_k^{r+1}$is Tate.

If two of the $E_i^{r+1}$ are exceptional divisors of $\on{Bl}_{\mc D_{\leqs
0}^r}\left(\on{Bl}_{\mc D^{\leqs r}}X\right)\ra \on{Bl}_{\mc D^{\leqs r-1}}X$
then the intersection is empty because the corresponding strata $D_i^r$ and
$D_j^r$ have an empty intersection (they have been separated at a previous
stage). 

Hence at most one of the $E_i^{r+1}$ is an exceptional divisor coming from
the last blow-up, and we can suppose that the strata $D_i^{r+1}, E_1^{r+1},
\ldots , E_{k-1}^{r+1}$ come from strata at the previous stage
$D_i^{r}, E_1^{r},\ldots , E_{k-1}^{r}$.
\begin{itemize}
 \item Suppose that $E_k^{r+1}$ is the proper transform of an exceptional
divisor $E_k^r$ in $\mc E^r$.  The subvariety $Y=D_i^r \cap E_1^r \cap \cdots
E_k^r$ is Tate by the induction hypothesis and its proper transform is 
\[
 \on{Bl}_{D_j^r \cap Y ; rank(D_j)=r} Y
\]
which is a Tate variety ($D_j^r \cap Y $ is either empty or Tate and $Y$ is
Tate). On the other side the proper transform of $Y$ is the intersection
$D_i^{r+1}\cap E_1^{r+1}\cap
\cdots \cap E_{k}^{r+1}$, which is therefore Tate.
\item Suppose that $E_k^r$ is the exceptional divisor coming from the blow-up
of $\on{Bl}_{\mc D^{\leqs r-1}}X$ along $D_j^r$. Let $Y$ be the intersection
$D_i^r \cap E_1^r \cap \cdots
E_{k-1}^r$. Then $D_j^r \cap Y$ is either empty or a Tate variety In the first
case the intersection $D_i^{r+1}\cap E_1^{r+1} \cap \cdots \cap E_k^{r+1}$ is
empty. In the latter case we have 
\[
 D_i^{r+1}\cap E_1^{r+1} \cap \cdots \cap E_k^{r+1}= \p(N_Y Y\cap D_j^r)
\]
which is Tate.
\end{itemize}
\end{proof}

\subsection{The space $X_n$ and some of its properties} \label{conXn}

Let $n$ be an integer greater than $2$ and let $x_1,
\ldots,x_n$ be the natural coordinates on $\A^n$. We define the divisors $A_I$,
$B^0_i$, $B^1_i$, $A_n$, $B_n$, $D^1_n$, $D^0_n$ and $D_n$ as follows:
\begin{itemize}
 \item for all non empty subsets $I$ of $\ent{1,n}$, $A_I$ is the
   divisor defined by  
\[1-\prod_{i\in I} x_i =0;\]
\item for all $i \in \ent{1,n}$, $B^0_i$ is the divisor defined by $x_i=0$;
\item for all $i \in \ent{1,n}$, $B^1_i=A_{\{i\}}$ is the divisor defined by
$1-x_i=0$;
\item $B_n$ is the union $ (\bigcup_i B^0_i)\bigcup(\bigcup_i B^1_i) $;
\item $A_n$ is the union $\bigcup_{I\subset \ent{1,n}; |I|\geqs2} A_I$;
\item $D^1_n$ is the union $\bigcup_{I\subset \ent{1,n}; I\neq \emptyset} A_I$ ;
\item $D^0_n$ is the union $\bigcup_i B^0_i$;
\item $D_n$ is the union $D^0_n \bigcup D^1_n$.
\end{itemize}
 
\begin{rem}
 The divisor $B_n$ is the Zariski closure of the boundary of the real cube
$C_n=[0,1]^n$ in $\A^n (\R)$.
\end{rem}

As the divisor $D_n$ is not normal crossing, we would like to find a
suitable succession of blow-ups that will allow us to have a normal
crossing divisor $\wh{D}_n$ over $D_n$. In order to achieve this we
first need the following remark and lemmas.

\begin{rem}\label{remtrans}
 Let $I$ be a non-empty subset of $\ent{1,n}$ and $x=(x_1,\ldots,x_n)$ a point
in
$A_I$. Then the normal vector of $A_I$ at the point $x$ is 
\begin{align}
 n^{A_{I}}_{|x}=\sum_{i \in I} \frac{1}{x_i}\on{d}\!x_i. \label{vecnor}
\end{align}
Therefore, if $I$ and $J$ are two distinct non-empty subsets of $\ent{1,n}$,
the
intersection of $A_I$ and $A_J$ is transverse.
\end{rem}

\begin{lem}\label{lemfondhyp} Let $I_1, \ldots, I_k$ be $k$ subsets of
$\ent{1,n}$ and $X$ the intersection $A_{I_1}\cap \cdots \cap
A_{I_k}\subset \A^n$. Then, there exist non negative integers $r$ and
$s$ with $r>0$, $s+r \leqs n$   and
integers $c_1,\ldots, c_r$ such that $X$ is isomorphic to 
\[
\A^s\times \Gm^{n-s-r}\times \prod_{i=1}^{r}\{x^{c_i}=1\}.
\]
\end{lem}
\begin{proof} If $|I_1 \cup \cdots \cup I_k|=a<n$ then $X$ is
isomorphic to $(A_{I_1}'\cap \cdots \cap A_{I_k}')\times
\A^{n-a}\subset \A^a\times \A^{n-a}$, where the $A_{I_i}'$ are defined
by the same equations, $1-\prod_{j\in I_i}x_j=0$, that define
$A_{I_i}$ but viewed in $\A^a$ instead of $\A^n$. Putting $s= n-a$, it
is enough to show that we have 
\[
X'=(A_{I_1}'\cap \cdots \cap A_{I_k}') \simeq \Gm^{a-r}\times
\prod_{i=1}^{r}\{x^{c_i}=1\}. 
\]
We now assume that $|I_1 \cup \cdots \cup I_r|=n$.

For any tuple $\lambda=(\lambda_1, \ldots, \lambda_n)$ of integers and any $x$ in $\A^n$,
let $x^{\lambda}$ denote the product
$\prod_{j=1}^nx_j^{\lambda_j}$. For $i$ in $\ent{1,n}$, let
$a_i=(a_{i1},\ldots, a_{in})$ be the element of $\Z^n$ defined by 
\[
\forall j, \, 1\leqs j \leqs n \quad a_{ij}=\delta_{I_i}(j) 
\]
where $\delta_{I_i}$ is the characteristic function of the set
$I_i$. Using these notations, $X$ is defined by the equations
\[
x^{a_1}=\cdots=x^{a_k}=1.
\]

Let $L$ be the submodule of $\Z^n$ spanned by $a_1,\ldots, a_n$, and
let $r$, $r \leqs k$, be its rank. For $\lambda$ in $L$, writing
\[
\lambda=\alpha_1 a_1 +\cdots \alpha_k a_k,  
\]
we see that for any $x$ in $\A^n$ we have 
\[
x^{\lambda}= \prod_{i=1}^n \left(x^{a_i}\right)^{\alpha_i}.
\]
In consequence, $x$ is in $X$ if and only if for all $\lambda $ in $L$
one has $x^{\lambda}=1$.

The module $L$ being a submodule of the free $\Z$-module $\Z^n$, we
have a basis $f_1, \ldots, f_n$ of $\Z^n$ and integers $c_1, \ldots,
c_r$ such that 
\[
L=f_1 \cdot c_1\Z \oplus \cdots \oplus f_r \cdot c_r\Z.
\] 

As an element $x$ of $\A^n$ is in $X$ if and only if
\[
\forall \lambda \in L, \, x^{\lambda}=1, 
\] 
we deduce that $x$ is in $X$ if and only if
\[
\left(x^{f_1}\right)^{c_1}=1, \quad \ldots, \quad \left(x^{f_r}\right)^{c_r}=1
\]
and $X$ is defined by the above equation.

Let $e_1,\ldots, e_n$ be the canonical basis of $\Z^n$ and $\vp$ be
the isomorphism of $\Z^n$ sending each $f_i$ to $e_i$ for $i$ in
$\ent{1,n}$. Let $(\vp{ij})_{\substack{1\leqs i \leqs n\\ 1\leqs i
    \leqs n}}$ denote the matrix of $\vp$ in the canonical basis. The
morphism $\vp$ induces a morphism $\tilde{\vp}$ from $\Gm^n$ to $\Gm^n$
defined on the coordinates by 
\[
\tilde{\vp}(x_j)= \prod_{i=1}^{n}x_i^{\vp_{ij}},
\]
such that $\tilde{\vp}$ sends $X$ to the subvariety $\tilde{X}$ defined by 
\[
x^{\vp(c_1f_1)}=1, \quad \ldots, \quad x^{\vp(c_rf_r)}=1. 
\]
As $\vp(c_if_i)= c_ie_i$ for all $i$ in $\ent{1,n}$, $\tilde{X}$ is
in fact defined by  
\[
x_1^{c_1}=1, \quad \cdots, \quad x_r^{c_r}=1.
\]

The morphism $\vp$ being invertible, $\tilde{\vp}$ is an isomorphism
and $X$ is isomorphic to 
\[
\Gm^{n-r}\times \prod_{i=1}^{r}\{x^{c_i}=1\}.
\]
\end{proof}
\begin{lem}\label{lemhypnor}
Let $I_1, \ldots, I_k$ be $k$ subsets of $\ent{1,n}$ and $X$ the
intersection $A_{I_1}\cap \cdots \cap A_{I_r}\subset \A^n$. Then, the
normal bundle $N_{\A^n}X$ is spanned by the normal bundle
$N_{\A^n}A_{I_1}, \ldots, N_{\A^n}A_{I_k}$.
\end{lem}
\begin{proof}
As in the proof of the previous lemma, it is enough to suppose that
$|I_1 \cup \cdots \cup I_k|=n$.

Each of the $N_{\A^n}A_{I_i}$ is a subbundle of $N_{\A^n}X$. Thus, as
$X$ is smooth, checking that we have the equality of dimensions is
enough. Using equation \eqref{vecnor}, we see that at a point $x$ of
$X$ the dimension of the vector space 
$\on{Vect(n_{|x}^{A_{I_1}}, \ldots ,n_{|x}^{A_{I_k}})}$ is equal to
the rank of the matrix 
\[
M= \left(\frac{1}{x_j}\delta_{I_i}(j) \right)_{
\substack{ 1 \leqs i \leqs k\\
   1 \leqs j \leqs n}
}
\] 
where $\delta_{I_i}(j)$ is the characteristic function of the set
$I_i$.  The rank of $M$ is that of the matrix 
$\left(\delta_{I_i}(j) \right)_{
\substack{ 1 \leqs i \leqs k\\
   1 \leqs j \leqs n}
} $, that is the rank of the $\Z$-module $L$ spanned by $a_1, \ldots ,
a_k$ defined in the previous proof. By the proof of Lemma
\ref{lemfondhyp}, the rank of $L$ is the codimension of $L$, which
completes the proof of lemma \ref{lemhypnor}. 
\end{proof}

\begin{lem}\label{D1An}
 Let $\mc D^1_n$ be the poset (for the inclusion) formed by all the 
 irreducible components of all possible
intersections of divisors $A_I$. Then the poset $\mc D^1_n$ satisfies the
conditions (1), (2) and (3) of theorem \ref{blsq}.
\end{lem}
\begin{proof}
 The intersection condition (3) follows from the definition of $\mc 
 D^1_n$. From Lemma \ref{lemfondhyp}, any possible intersection
 $X=A_{I_1}\cap \cdots \cap A_{I_k}$ is isomorphic to $\A^s\times
 \Gm^{n-s-r}\times \prod_{i=1}^{r}\{x^{c_i}=1\}$ for some non negative 
 integers $r$ and $s$ and integers $c_i$, thus $X$ is smooth and
  its irreducible components are all smooth.  

Let $S_1$ and $S_2$ be two elements of $\mc D^1_n$. To show that $S_1$ 
and
$S_2$ meet cleanly, it is enough to show that  the normal bundle
of the intersection is spanned by the normal bundles of $S_1$ and
$S_2$, that is 
\[N_{\A^n}(S_1 \cap S_2)=N_{\A^n}(S_1) +N_{\A^n}(S_2).\] 
As $S_1$ and $S_2$ are intersections of some $A_I$, it is enough to show that 
the normal bundle of $A_{I_1} \cap \cdots \cap A_{I_k}$ is spanned by the
normal vectors of the $A_{I_j}$ and that is ensured by lemma \ref{lemhypnor}.
%
\end{proof}

Applying the construction of theorem \ref{blsq} with $\mc D= \mc D^1_n$
and $X=\A^n$ leads to a variety $X_n \st{p_n}{\ra} \A^n$,
which results from successive blow-ups of all the strata of $\mc D^1_n$ such
that the preimage $\wh{\mc D^1_n}$ of $\mc D_n^1$ is a normal crossing
divisor. We will write $\wh{D^1_n}$ to denote the preimage of $D^1_n$.

\begin{lem}\label{B0ncd}
 Let $\wh{D^0_n}$ be the proper transform in $X_n$ of the divisor $D^0_n$.
Then $\wh{D}_n=\wh{D^1_n} \bigcup \wh{D^0_n}$ is a normal crossing divisor.
\end{lem}
\begin{proof}
Let $I$ be a non-empty subset of $\ent{1,n}$, $\wh{B^0_I}$ (resp. $B^0_I$) be
the intersection in
$X_n$ (resp. $\A^n$) of divisors $\{x_i=0\}$ for $i$ in $I$. And let
$\wh{S_1}, \ldots, \wh{S_k}$ be strata of
$\wh{\mc D_1^n}$ such that the intersection of the $\wh{S_i}$ is non-empty. We
want to show that there is a neighbourhood $V$ of $\wh{B^0_I} \bigcap
\wh{S_1} \bigcap
\cdots \bigcap \wh{S_k}$ such that $V \cap\wh{D}_n $ is  normal
crossing. By theorem
\ref{blsq}, the $\wh{S_i}$ come from strata of $\mc D^1_n$, $S_1\subset
\cdots \subset S_k$. As the intersection of the $\wh{S_i}$'s with $\wh{B^0_I}$
is non-empty, the intersection of $B^0_I$ with $S_1$ is non-empty.  There
exist non-empty subsets $I_1, \ldots,I_l$  of $\ent{1,n}$ such that $S_1=
A_{I_1}\cap \cdots \cap A_{I_l}$.

As $B^0_I \bigcap S_1$ is non-empty, we have 
\[I\bigcap (I_1\bigcup \cdots
\bigcup I_l)=\emptyset.
\]

Then, in $\A^n$, we have a neighbourhood $V_0$ of $B^0_I \bigcap S_1$
isomorphic to a product $\A^{d}
\times \A^{|I|}$ with $d=n-|I|$:
\[
\begin{array}{ccc} \A^d \; &\times & \; \A^{|I|} \\
\cup & & \cup \\
{\tilde{D}^1_{d}} & & \bigcup_{i \in I}\tilde{B}^0_i,
\end{array}
\]
where $\tilde{B}^0_i$ is the hyperplane corresponding to $\{x_i=0\}$
inside $\A^{|I|}$.
 
Lifting this neighbourhood to $\wh{V_0}$ in $X_n$, it becomes isomorphic to
$X_d \times \A^{|I|}$ with $\wh{D}^1_d \subset X_d$.
Then, for any $\wh{S_i}$ there is a stratum $\wh{S_i^d}$ of $\wh{D}^1_d$ such
that $\wh{V_0}\cap \wh{S_i}\simeq \wh{S_i^d}\times \A^{|I|}$. As the
$\wh{S_i^d}$'s give a
normal
crossing divisor in $X_d$ by Theorem \ref{blsq}, $\wh{V_0}$ gives
the neighbourhood of $\wh{B^0_I}
\bigcap
\wh{S_1}
\bigcap \cdots \bigcap \wh{S_k}$ such that $ V \cap \wh{D}_n$ is a
normal crossing divisor in $X_n$.

\end{proof}

\begin{defn}
Let $\wh{B}_n$ denote the preimage of $ B_n$ and $\wh{A}_n$
 the divisor $\wh{D}_n\sm \wh{B}_n$.
\end{defn}

\begin{rem}
 The divisors $\wh{A}_n$ and $\wh{B}_n$ do not share any irreducible
components and are both normal crossing divisors.
\end{rem}

Let $\wh{C}_n$ be the preimage of $C_n=[0,1]^n$ in $X_n$ and $\ol{\wh{C}_n}$
its
closure. Then $\wh{B}_n$ is the Zariski closure of the boundary of
$\ol{\wh{C}_n}$, and there is a non-zero class
\begin{align}[\wh{C}_n] \in \Gr_0^W\HH^n(X_n,\wh{B}_n).
\end{align}

%
%
%
%
%

If $I$ is a subset of $\ent{1,n}$, we define $F_I$ and $G_I$ to be the
functions 
\[\begin{array}{l} G_I: (x_1,\ldots, x_n)\longmapsto \prod_{i\in
    I}x_i\\
F_I : (x_1,\ldots, x_n)\longmapsto 1-\prod_{i\in
    I}x_i.\end{array}\]  
\begin{defn} A flag $\mc F$ of $\ent{1,n}$ is a collection of
  non-empty distinct subsets $I_j$ of $\ent{1,n}$ such that $I_1
  \subsetneq \ldots
  \subsetneq I_r$. The length of the flag $\mc F$ is the
  integer $r$ and we may say that $\mc F$ is an $r$-flag of
  $\ent{1,n}$.  
A flag of length $n$ will be a maximal flag.
A distinguished $r$-flag $(\mc F,i_1,\ldots i_p)$ will be a flag $\mc F$ of
length $r$ together with elements $i_1<\ldots< i_p$ of
$\ent{1,r}$.
\end{defn} 

\begin{defn}
Let $(\mc F,i_1,\ldots i_p)$ be a distinguished $r$-flag of $\ent{1, n}$. Let
$\Om_{i_1,\ldots, i_p}^{\mc F}$ denote the differential form of
$\Om_{log}^{\bullet}(\A^n\sm D_n)$ defined by 
\[\Om_{i_1,\ldots i_p}^{\mc F} =\bigwedge_{j=1}^{r}\on{d}\log(g_j)
\]
where 
\[g_j=\left\{\begin{array}{l}
       F_{I_j} \mx{if }j\in \{i_1,\ldots,i_p \} \\
	G_{I_j} \mx{otherwise}.
      \end{array}\right.
\]

Let $\mb k=(k_1,\ldots, k_p)$ be a tuple of positive integers with
$k_1\geqs 2$ such that $k_1+\cdots+ k_p=n$, and let $s$ be a permutation of
$\ent{1,n}$. We define a differential form $\Om_{\mb k, s}\in
\Om_{log}^n(\A^n\sm D_n)$ by 
\[\Om_{\mb k,s}=\f{k}{n}(x_{s(1)},\ldots,x_{s(n)})\dd x_1\w \cdots \w
\dd x_n.\] 
\end{defn}
\begin{rem}\label{flagom}
Let $\mb k$ and $s$ be as in the previous
definition. We associate to the pair $(\mb k,s)$ the maximal
distinguished flag $\mc
(F_k, i_1, \ldots, i_p)$ defined by $I_i=\{s(1), \ldots,s(i)\}$ and 
$i_j=k_1+\cdots+k_j$ where $j$ varies from $1$ to $p$. Then we can see that
there exists an integer $r_s$ such that
\[\Om_{\mb k, s}=(-1)^{r_s}\Om_{i_1,\ldots i_p}^{\mc F_k}.\]
\end{rem}

\begin{defn}
We shall write $\om_{i_1, \ldots, i_p}^{\mc F}$ and $\om_{\mb k ,s}$ for the
pull back on $X_n \sm \wh{D}_n$ of the forms $\Om_{I_1,
  \ldots, I_p}$ 
and $\Om_{\mb k ,s}$, respectively.
\end{defn}

\begin{prop}\label{stformlog}
If $(\mc F,i_1,\ldots,i_p)$ is a maximal flag of $\ent{1, n}$  such that
$i_1\geqs 2$ and $i_p=n$ then:
\begin{itemize}
\item  The 
divisor of singularities $A^{\mc F}_{i_1,\ldots,i_p}$ of
$\Om_{i_1,\ldots,i_p}^{\mc F}$ is $A_{I_{i_1}}\cup \cdots \cup A_{I_{i_p}}$.
\item  The 
divisor of singularities $\wh{A}^{\mc F}_{i_1,\ldots,i_p}$ of
$\om_{i_1,\ldots,i_p}^{\mc F}$ lies in
$\wh{A}_n$.Thus, the divisor of singularities of $\om_{\mb k,s}$ lies
  in $\wh{A}_n$. 
\end{itemize} 

 Moreover, let $(\mc F,i_1,\ldots,i_p)$ and $(\mc F',i'_1,\ldots, i'_q)$
 be two distinguished flags such that the length of $\mc F$
 (resp. $\mc F'$) is $i_p$ (resp $i'_q$)
  with $|I_{i_1}|\geqs 2 $ (resp. $|I'_{i'_1}|\geqs 2$) and suppose that the
  sets $I_{i_p}$ and 
 $I'_{i'_q}$ form a partition of $\ent{1,n}$. Then the divisor of
 singularities of $\om_{i_1,\ldots,i_p}^{\mc F}\w \om_{i'_1,\ldots
   ,i'_q}^{\mc F'}$ lies in $\wh{A}_n$.
\end{prop}
Let $(\mc F, i_1\subsetneq \ldots \subsetneq i_p)$ be a flag as in the
previous proposition.

It is straightforward to see that $A^{\mc F}_{I_{i_1}, \ldots
  I_{i_p}}$ is $A_{I_{i_1}}\cup \cdots \cup A_{I_{i_p}}$. 
The following lemma due to Goncharov can easily be modified to fit into
our situation. 
\begin{lem}[{\cite{PMMGon}[lemma 3.8]}] Let $Y$ be a normal crossing
  divisor in a smooth variety $X$ and $\om\in\Om_{log}^n(X\sm Y)$. Let
  $p: \wh{X}\lra X$ be the blow-up of an irreducible variety
  $Z$. Suppose that the generic point of $Z$ is different from the
  generic points of strata of $Y$. Then $p^*\om$ does not have a
  singularity at the special divisor of $\wh{X}$.
\end{lem}
The modified version we use here is given in the following statement.
\begin{lem} Let $Y$ be a normal crossing divisor in $\A^n$ and
  $\om\in\Om_{log}^n(\A^n\sm Y)$. Let $p_n : X_n \ra \A^n$ be the map
  of our previous construction. Suppose that the generic points of the strata
of
  $B_n$ that are blown up in the construction of $X_n$ are different
  from the generic points of strata of $Y$. Then $p_n^*\om$ does not
  have singularities at the corresponding exceptional divisors in $\wh{B}_n$.
\end{lem}
It is enough to check that
the divisor of singularities of $\Om_{i_1, \ldots, i_P}^{\mc F}$ is a a normal
crossing divisor, and that none of its strata is a blown up strata of
$B_n$.

The divisor of singularities of $\Om_{i_1,...,i_p}^{\mc F}$ is
$A_{I_{i_1}}\cup\cdots \cup A_{I_{i_p}}$, and to show it is a normal crossing
divisor, it is enough to show that the normal vectors of the $A_{I_{i_j}}$
at any intersection of some of them are linearly independent. The
normal vector of $A_{I_{i_j}}$ is $\sum_{i\in I_{i_j}}1/x_i \dd x_i$, and as we
have $I_1\subsetneq I_2 \subsetneq \ldots \subsetneq I_p$, they are
linearly independent.

We now have to show that none of the strata of $B_n$ that are blown up
in the construction of $X_n$ are exactly some strata of $A_{I_1}\cup
\cdots A_{I_p}$. Let $S$ be such a stratum of $B_n$ of codimension
$k$. The stratum $S$ is defined by the equations $x_{r_1}=1, \ldots,
x_{r_k}=1$. If $I_S$ denotes the set $\{r_1,\ldots, r_k\}$, then for
any subset $I$ of $\ent{1, n}$, $S$ is included in  $A_I$ if and
only $I$ is included in $I_S$. As $I_i\subset I_{i'}$ for $i<i'$, if
$S$ is included in a stratum $S_A$ of $A_{I_1,\ldots,I_p}$, that stratum
is of the form $A_{I_{i_1}}\cap\cdots \cap A_{A_{i_j}}$ with $j<k$ because
$|I_1|\leqs 2$. As a consequence, $S_A$ is of codimension at most $k-1$,
and $S$ cannot be a stratum of $A_{i_1,\ldots,i_p}^{\mc F}$. 

For the case of two distinguished flags, we use the same argument as in
the lemma, so the proposition \ref{stformlog} is proved.

\begin{prop}\label{boundivXn}
The divisor $\wh{A}_n$ does not intersect the boundary of $\wh{C}_n$ in
$X_n(R)$.
\end{prop}
\begin{proof}
Let $S$
be an irreducible codimension $1$ stratum of $\wh{B}_n$ containing an
intersection of some strata of $\wh{A}_n$ with the boundary of
$\ol{\wh{C}_n}$. As the divisor $A_n$ intersects the boundary of the real cube
$C_n$ only
on strata of $B_n$ that are of codimension at least $2$,  $S$ has to
be such that $p_n(S)$ is a stratum of
$B_n$ of codimension at least $2$. 

Using the symmetry with respect to the standard coordinates on $\A^n$,
we can suppose that $p_n(S)$ is defined in those coordinates by
$x_k=x_{k+1}=\ldots=x_n$. 

Starting from $\A^n$ and blowing up first the point
$x_1=x_2=\ldots=x_n=1$, then the edge $x_2=x_3=\ldots=x_n=1$ and after
that the plane $x_3=x_4=\ldots=x_n=1$ and so on, we obtain a variety
$\wt{p}_n:\wt{X}_n\ra \A^n$. There are natural local coordinates
$(s_1,\ldots,s_n))$ on
$\wt{X}_n$ such that the coordinates on $\A^n$ defined by $y_i=1-x_i$
satisfy:
\[y_1=s_1, \quad y_2=s_1s_2,\quad \ldots, \quad y_i=s_1s_2\cdots s_i,\quad
\ldots, \quad y_n=s_1s_2\cdots s_n.\]
In the $y_i$-coordinates, the stratum $x_j=x_{j+1}=\ldots=x_n=1$ is
given by
\[
y_j=y_{j+1}=\ldots = y_n=0
\] 
and its preimage in $\wt{X}_n$ is given by $s_j=0$.

For any permutation $s$ of $\ent{1,n}$, we could apply the same
construction, that is blowing up the point
$x_{s(1)}=x_{s(2)}=\ldots=x_{s(n)}=1$ then the edge
$x_{s(2)}=x_{s(3)}=\ldots=x_{s(n)}=1$ and so on, and obtain a variety
$\wt{p}_n^{s}: \wt{X}_n^s \ra \A^n$. The preimage of $D_n $ in
$\wt{X}_n^s$ will be denoted by $\wt{D}_n^s$, $\wt{B}_n^s$ will denote
the preimage of $B_n$ and
$\wt{A}_n^s$ is $\wt{D}_n^s \sm \wt{B}_n^s$. To prove that $\wh{A}_n$ does not
intersect the boundary of $\wh{C}_n$ in $X_n(\R)$, it is enough to show
that for any permutation $s$, $\wt{A}_n^s$ does not intersect, in
$\wt{X}_n^s(\R)$, the
boundary of the preimage of $C_n$. It is then enough  to show that the 
proper transforms of the divisors $A_I$ do not intersect the boundary of 
$\tilde{X}^s_n(\R)$, because it will then be the same for the irreducible 
components of their intersections as for the proper transforms of those 
components by the remaining blow-up used to reach $X_n$.  
By symmetry,  it is enough to show it when
$s$ is the identity map and then in the case of $\wt{X}_n$. Let
$\wt{C}_n$ be the
preimage of $C_n$ in $\wt{X}_n$ .

Let $A_I$ be a codimension $1$ stratum of $A_n$, where $I$ is the set
$\{i_0,\ldots,i_p\}$ and suppose that $i_0<\ldots<i_p$. We want to show
that the closure $\wt{A}_I$ of the preimage of $A_I\sm B_n$ in
$\wt{X}_n$ does not intersect the boundary of $\wt{C}_n$. The $k$-th symmetric
function will be denoted by $\sigma_k$ with the following convention
\[ \sigma_0=1, \qquad \sigma_k(X_1,\ldots,X_l)=0 \mx{ if }l>k.\]
The stratum
$A_I$ is defined in the $x_i$-coordinates by
$1-x_{i_0}\cdots x_{i_p}=0$ and in the $y_i$ coordinates by 
\begin{align}\label{eqAIAn}0 =\sum_{k=1}^{p+1} (-1)^{k-1}
  \sigma_k(y_{i_0} ,y_{i_1}\ldots,y_{i_p}).
\end{align}

Before giving an explicit expression of $\wt{A}_I$ with the $s_i$
coordinates, we define the set $J_0$ as $\{1,\ldots,i_0\}$ and the
sets $J_1,\ldots,J_p$ by 
\[ J_k=\{i_0+1,i_0+2,\ldots,i_k\}\]
for all $k$ in $\ent{1,p}$.

 For any subset $J$ of $\ent{1,n}$, $\pjs{}$ will denote the
 product $\prod_{j\in J}s_j$. We have the following relations
\[\begin{array}{l} y_{i_0}=\pjs{0} \qquad \mx{and}
\qquad \forall k\in \ent{1,p}, \quad y_{i_k}=\pjs 0 \pjs k.
\end{array}
\]

Using the variable change $y_i=s_1\cdots s_i$, the RHS of the equation
(\ref{eqAIAn}) can be written
\begin{align}\label{eqAIXt1}\sum_{k=1}^{p+1} (-1)^{k-1}
  \sigma_k(\pjs{0} ,\pjs 0 \pjs 1, \ldots,\pjs 0 \pjs{p}).\end{align}
For any indeterminate $\lambda$ and any $k$, one has
\[\sigma_k(\lambda,\lambda X_1, \lambda X_2,\ldots , \lambda
X_p)=\lambda^k(\sigma_{k-1}(X_1\ldots,X_p)+\sigma_k(X_1,\ldots,X_p)).\]
Then 
the expression (\ref{eqAIXt1}) is equal to

\begin{multline*} \pjs 0\bigg[
    1+\sigma_1(\pjs 1 ,\ldots,\pjs p)\\ +\sum_{k=1}^{p-1}\left((-1)^{k}(\pjs
    0)^{k}\left(\sigma_k(\pjs 1, \ldots, \pjs p) + \sigma_{k+1}(\pjs 1,
      \ldots, \pjs p) \right)\right)\\[-4mm]
\\
 +(-1)^p\sigma_p(\pjs 1, \ldots, \pjs p)\bigg].
\end{multline*}
The expression of $\wt{A}_I$ in the $s_i$-coordinates is then 
\begin{multline}\label{eqAIXt2} 0=
    1+\sigma_1(\pjs 1 ,\ldots,\pjs p) \\ + \sum_{k=1}^{p-1}\left((-1)^{k}(\pjs
    0)^{k}\left(\sigma_k(\pjs 1, \ldots, \pjs p) + \sigma_{k+1}(\pjs 1,
      \ldots, \pjs p) \right)\right)\\
\\
 +(-1)^p\sigma_p(\pjs 1, \ldots, \pjs p).
\end{multline}
The closure of $\wt{C}_n$ is given, in the $s_i$ coordinates, by $s_1
\in [0,1]$, and for any $i\in\ent{1,n}$ by $s_1\cdots s_i\in [0,1]$. It
is enough to look the intersection of $\wt{A}_I$ with codimension $1$
strata of the boundary of $\wt{C}_n$.

Suppose that $s_{i_0}=0$ for some $i_0 \in J_0$. Then the RHS of
(\ref{eqAIXt2}) becomes
\[1+\sigma_1(\pjs 1 ,\ldots,\pjs p)\] 
which is strictly positive if $s_i\geqs 0$ for any $i$. So the divisor
$\wt{A}_I$ does not intersect any component of the form $s_{i_0}=0$ for
$i_0$ in $J_0$. 

Then, we can suppose that $s_i\neq 0$ for all $i\in
J_0$ in order to study the intersection of $\wt{A}_I$ with the
boundary of $\wt{C}_n$, and the RHS of (\ref{eqAIXt2}) can be written
\[
\fr{1}{\pjs 0}\left(1-\prod_{j=1}^{p}\left(1-\pjs 0 \pjs
    j\right)\right) +\prod_{j=1}^{p}\left(1-\pjs 0 \pjs
    j\right).\] 
Suppose that a point $x=(s_1,\ldots,s_n)$ with $s_i>0$ for all $i$ in
$J_0$, lies in the closure of $\wt{C}$. That is, for all $i$ in
$\ent{1,n}$, the product $s_1s_2\cdots s_i$ is between $0$ and $1$ which 
means all the products $\pjs 0 \pjs j$ lie between $0$ and $1$ for $j$
in $\ent{1,p}$. Then one find the following inequalities
\[\begin{array}{c}
\ds 0 \leqs \fr{1}{\pjs 0}\left(1-\prod_{j=1}^{p}\left(1-\pjs 0 \pjs
    j\right)\right) \leqs \fr{1}{\pjs 0},\\[5mm]
\ds 0 \leqs \prod_{j=1}^{p}\left(1-\pjs 0 \pjs
    j\right)\leqs 1.
\end{array}
\]
Both terms cannot simultaneously be equal to $0$, thus $\wt{A}_I$ does 
not
intersect the boundary of $\wt{C}_n$ since the $s_i$ are strictly positive
for $i$ in $J_0$ and the proposition is proved.

\end{proof}
\subsection{An alternative definition for motivic MZV}
Both propositions \ref{stformlog} and \ref{boundivXn} lead to the
following theorem and to an alternative definition for motivic multiple
zeta values.
\begin{thm}Let $\mb k =(k_1,\ldots, k_p)$ be a tuple of integers with
  $k_1\geqs 2$ and $k_1+\ldots+k_p=n$, and let $s$ be a permutation of
  $\ent{1,n}$. Let $\wh{A}_{\mb k}^s$ be the divisor of singularities
  of the differential form $\om_{\mb k}^s$. Then there exists a mixed
  Tate motive  
\[\HH^n(X_n \sm \wh A_{\mb k}^s ; \wh{B}_n^{\wh A_{\mb k}^s}).
\]
The differential form $\om_{\mb k}^s$ and the preimage $\wh{C}_n$ of
the real $n$-dimensional cube in $X_n$ give two non zero elements
\[
[\om_{\mb k}^s]\in \Gr_{2n}^W\HH^n(X_n \sm \wh A_{\mb k}^s ;
\wh{B}_n^{\wh A_{\mb k}^s})\qquad \mx{and} \qquad  [\wh{C}_n] \in
\left(\Gr_{0}^W\HH^n(X_n \sm \wh A_{\mb k}^s ;
\wh{B}_n^{\wh A_{\mb k}^s})\right)^{\vee} \] 
The period of the $n$-framed mixed Tate motive
\[ \zeta^{fr., \mc M}(\mb k, s)=\left[ \HH^n(X_n \sm \wh A_{\mb k}^s ;
\wh{B}_n^{\wh A_{\mb k}^s});
  [\om_{\mb k}^s],[\wh{C}_n]\right]
\]
is equal to $\zeta(k_1,\ldots, k_n)$.

Moreover, let $(\mc F,i_1,\ldots,i_p)$ and $(\mc F',i'_1,\ldots, i'_q)$
be two distinguished flags such that the length of $\mc F$ (resp. $\mc
F'$) is $i_p$ (resp. $i'_q$)
  with $|I_{i_1}|\geqs 2 $ (resp. $|I'_{i'_1}|\geqs 2$) and the sets  $I_{i_p}$,
 $I'_{i_q}$ form a partition of $\ent{1,n}$ and let  $\wh A^{\Fc | \Fc'}_{i_1,
\ldots, l_p |i'_1, \ldots, i'_q}$ be the divisor of singularities of
$\om_{i_1,\ldots,i_p}^{\mc F}\w \om_{i'_1,\ldots,i'_q}^{\mc F'}$. There exists
an $n$-framed mixed Tate motive 
\[
 \zeta^{fr., \mc M}(\Fc, i_1, \ldots, i_p | \Fc', i'_1,
\ldots, i'_q)=\HH^{n}(X_n \sm \wh A^{\Fc | \Fc'}_{i_1,
\ldots, l_p |i'_1, \ldots, i'_q} ; \wh B_n^{\wh A^{\Fc | \Fc'}_{i_1,
\ldots, l_p |i'_1, \ldots, i'_q}}),
\]
the frames being given by $[\om_{i_1,\ldots,i_p}^{\mc F}\w
\om_{i'_1,\ldots,i'_q}^{\mc F'}]$ and $[\wh C_n]$.
\end{thm}
\begin{proof}
We want to apply theorem 3.6 in \cite{PMMGon} to our particular case. As
$\wh{D}_n$ is a normal crossing divisor and as proposition \ref{boundivXn}
ensures that $\wh{A}_n$ does not intersect $[\wh{C}_n]$, using Proposition
\ref{stformlog}, the only thing that
remains to show is that we have a Tate stratification of $X_n$, which is
ensured
by Lemma \ref{Tatestrat}.

 The computation of the period follows from
the fact that integrating over $\wh{C}_n$ is the same as
integrating over the real cube.
\end{proof}

The key to prove Lemma \ref{Tatestrat} is Lemma \ref{lemfondhyp}
from which we deduce the following lemma.
\begin{lem}\label{lemhyperboles}
Let $I_1, \ldots, I_k$ be $k$ subsets of $\ent{1,n}$ and let $X$ be
the intersection  
$A_{I_1}\cap \cdots \cap A_{I_k}\subset \A^n$. Then $X$ and its irreducible 
components are Tate varieties.
\end{lem}
\begin{proof} Using Lemma \ref{lemfondhyp}, we have non negative integers
  $r$ and $s$ and integers $c_1, \ldots , c_r$ such that there exists
  an isomorphism $f$ 

\[
 X \xrightarrow[\hspace{1cm}]{f} 
\A^s\times \Gm^{n-s-r}\times \prod_{i=1}^{r}\{x^{c_i}=1\} .
\] 

Moreover, there is a one to one map between the set of the irreducible
components of $X$ and the set of those of
$\prod_{i=1}^{r}\{x^{c_i}=1\}$ ; the irreducible components
of $X$ are thus disjoint.

%
%

We conclude that the motive of $X$ is a direct sum of
Tate motives, in other words $X$ is a Tate variety. The irreducible
components of $X$ being disjoint, each is a Tate variety.
\end{proof}

\begin{lem}\label{Tatestrat}
 The divisor $\wh D_n=\wh B_n^0\cup \wh D_n^1$ provides $X_n$ with a Tate
stratification.
\end{lem}
\begin{proof}

We first need to show that all the strata of $\wh D_n^1$ and $X_n$ are Tate, but
using
Proposition \ref{Tatebl}, it is enough to show that all the strata of $
D_n^1$ are Tate ($\A^n$ being Tate). A stratum $A_{I_1} \cap \cdots
\cap A_{I_k}$ of
$D_n^1$ is a Tate variety by Lemma \ref{lemhyperboles}. So $X_n$ and all the strata of $\wh D_n^1$ are Tate.

Note that the previous discussion tells us that for any $k\geqs 2$, $X_k$ and
all the strata of $\wh D_k^1$ are Tate varieties.

Let $\wh S$ be the intersection of certain codimension $1$
strata of $\wh B_n^0$; it is the proper transform of the corresponding
intersection, say $S=\cap_{j\in J}\{x_j=0\}$ for some $J\subset \ent{1,n}$, in
$B_n^0$. That is, $\wh S$ is isomorphic to 
\begin{gather}\label{protrB0}
 \on{Bl}_{S\cap D \, : \, D \in \mc D_n^1} S.
\end{gather}
The intersection $S$ is isomorphic to $\A^d$ for $d=n-|J|$ and hence is Tate,
and
if $I$ is a subset of $\ent{1,n}$, then $S\cap A_I$ is either empty ($I\cap
J\neq \emptyset$) or, if $I\cap J=\emptyset$, isomorphic to the subvariety of
$\A^d$ given by $\{1-\prod_{i \in I}x_i=0\}$ (up to renumbering). Thus, the
proper
transform $\wh S$ is isomorphic to $X_d$, which is Tate by the discussion above.

Now, if $\wh S_i$ is some irreducible codimension $1$ stratum of $\wh D_n^1$
that has a non-empty intersection with $\wh S$, then, as $\wh S_i$ is the
exceptional divisor of some of the blow-ups in the construction of $X_n$, this
intersection $\wh S \cap \wh S_i$ is the exceptional divisor in the blow-up
sequence \eqref{protrB0} that leads to $\wh S$. As a consequence, the
intersection $\wh S \cap \wh S_i$ is isomorphic to some irreducible stratum
 of $\wh D_{d}^1$ in $X_d$ and we can conclude that any possible
intersection of strata in $\wh D_n^1$ with $\wh S$ is isomorphic to an
intersection of strata in $\wh D_d^1$ inside $X_d\simeq \wh S$, and so is
Tate by the above discussion.
\end{proof}

\subsection{Motivic Stuffle}
Let $\mb k= (k_1, \ldots,k_p)$ and $\mb l=(l_1, \ldots,l_q)$ be respectively
a $p$-tuple and a $q$-tuple of integers with $k_1,l_1 \geqs 2$, $\sum k_i=n$
and $\sum l_j=m$. In this section, as in section \ref{serstusec} and
\ref{secstuint}, if $\sigma$
is a term of the formal sum $\mathbf{k}*\mathbf{l}$ with all coefficients being
equal to $1$, we will write $\sigma \in
\on{st}(\mathbf{k},\mathbf{l})$. The map $\delta$ defined in Proposition
\ref{stmod} extends to: 
%
$$\begin{array}{ccc}\mnb{n+m}&\xrightarrow[\hspace{1cm}]{\delta}&
  \mnb{n}\times \mnb{m}.
\end{array}$$

Let $\As k $ (resp. $\As l$) be the divisor of singularities of the
meromorphic differential form $\oms k$ on $\mnb{n}$ (resp. $\oms l$ on
$\mnb{m}$) given in simplicial coordinates by $\om_{\ol k}$ (resp. 
$\om_{\ol l}$) (cf. \ref{Kontform}) and given in the cubical coordinates by 
$\f{k}{p}$ (resp.
$\f{l}{q}$). For all $\sigma$ in $\on{st}(\mb k , \mb l )$, let 
$A_{\sigma}$ be
the divisor of singularities of the form $\om_{\sigma}$. As in section
\ref{sec:motshu}, let $\Phi_n$, $\Phi_m$ and $\Phi_{n+m}$ denote
respectively the standard cells in $\mnb{n}(\R)$, $\mnb{m}(\R)$ and
$\mnb{n+m}(\R)$ and $B_n$, $B_m$ and $B_{n+m}$ be the Zariski closure of 
the boundary of  $\Phi_n$, $\Phi_m$ and $\Phi_{n+m}$ respectively.
 
\begin{prop}\label{stm} We have an equality of framed motives:
\begin{multline*}
\left[\HH^n\pa{\mnb{n}\sm \As k;B_n^{\As k} }; [\oms k];
  [\Phi_n] \right]\cdot \left[\HH^m\pa{\mnb{m}\sm\As l;B_m^{ \As l
    } } ; [\oms l]; [\Phi_m]\right]=\hspace{.6cm} ~\\[2mm]
\sum_{\sigma\in\on{st}(\mb k,\mb l)} \left[
  \HH^{n+m}\pa{\mnb{n+m}\sm A_{\sigma};B_{n+m}^{ A_{\sigma}}}; [\om_{\sigma}];
[\Phi_{n+m}]\right].
\end{multline*}
\end{prop}
\begin{proof}
 Let $A_0$ be the Zariski closure of $\partial \mnb{n+m} \sm B_{n+m}$, $B_{n,m}$  the Zariski
closure of the boundary of $\Phi_n \times \Phi_m$ and $A'$ the boundary of
$\left(\mnb{n} \sm \As k\right) \times \left(\mnb{m} \sm \As l\right)$. As the
map $\delta$ maps $B_{n+m}$ onto $B_{n,m}$, we have an induced map 
\[
 \delta : \left(\mnb{n+m}\sm A_0 ; B_{n+m}^{A_0}\right)\lra \left(\left(\mnb{n}
\sm \As k\right) \times \left(\mnb{m} \sm \As l\right) ; B_{n,m}^{A'}\right).
\]
Using the K\"unneth formula, we have maps of mixed Tate motives
\begin{multline}
\HH^n\pa{\mnb{n}\sm \As k;B_n^{\As k} } \otimes \HH^m\pa{\mnb{m}\sm\As l;B_m^{
\As l} } \lra \\ \HH^{n+m}\pa{(\mnb{n}\sm \As k)\times (\mnb{m}\sm\As l) ;
B_{n,m}^{A'}} 
\lra \\ \HH^{n+m}\pa{\mnb{n+m}\sm A_0; B_{n+m}^{A_0}} 
\end{multline}
which are both compatible with the respective frames $[\oms k]\ot[\oms l] ;
[\Phi_n]\ot [\Phi_m]$ , $[\oms k \w \oms l] ; [\Phi_n \times \Phi_m]$ and
$[\delta^*(\oms k \w \oms l )] ; [\Phi_{n+m}]$. 

We now need to show that 
\begin{multline*}
\left[ \HH^{n+m }\pa{\mnb{n+m}\sm A_0 ; B_{n+m}^{A_0}} ; [\delta^*(\oms k \w
\oms l )] , [\Phi_{n+m}] \right]= \\ 
\sum_{\sigma\in\on{st}(\mb k,\mb l)} \left[
  \HH^{n+m}\pa{\mnb{n+m}\sm A_{\sigma};B_{n+m}^{ A_{\sigma}}}; [\om_{\sigma}];
[\Phi_{n+m}]\right].
\end{multline*}
As $A_{\sigma}$ is included in $A_0$, using lemma \ref{inclmot} it is
enough to
prove the previous equality with $A_0$ instead of $A_{\sigma}$ in the RHS. The
two following lemmas tell us that it is enough to work with $X_{n+m}$ (cf.
section \ref{conXn}) instead of $\mnb{n+m}\sm A_0$.
\begin{lem}
 Let $r \geqs 2$ be an integer and let $\tilde{\delta_r} : \mnb r \ra
 (\p^1)^r$ be the map given on the open set by  
\begin{multline*}
 (0,z_1, \ldots, z_r,1,\infty) \longmapsto
(0,z_1,z_2,\infty)\times(0,z_2,z_3,\infty) \times \cdots \\ \times
(0,z_{r-1},z_r,\infty) \times (0,z_r,1,\infty).
\end{multline*}
 Let $A_r$ be the union of the codimension $1$ irreducible components
 of $\partial \mnb{r}$ that are mapped by $\tilde{\delta_r}$ into
 $(\p^1)^r\sm \A^r$. 

Then, $A_r\subset A_0$ and there exists a sequence of flags $\mc F_1,
\ldots, \mc F_N$ of elements 
of $\mc D_r^1$ (Lemma \ref{D1An}) satisfying conditions of Corollary
\ref{flblsq} such that
\begin{align}\label{BLXnMn}
 X_{r} = \on{Bl}_{\mc F_N, \ldots, \mc F_1}\A^{r} \st{\alpha_r}{\lra}
\mnb{r}\sm A_r=\on{Bl}_{\mc F_r, \ldots, \mc F_1  }\A^{r}
\st{\tilde{\delta_r}}{\lra} \A^r.
\end{align}

\end{lem}
\begin{proof} 
The map $\tilde{\delta_r}$ is given in cubical coordinates on $\mnb r$
by $x_i=u_i$, where the $x_i$ 
denote the standard affine coordinates on $(\p^1)^r$. It maps $B_r$
into hyperplanes $x_i=0$ or $x_i=1$. 

The induced map $\mnb{r}\sm A_r \ra \A^r$ is the blow-up along the strata
\begin{gather}\{x_i=x_{i+1}=\ldots=x_j=1\}
\end{gather}
 which are all elements of $\mc
D_r^1$. 

The beginning $\mc F_1, \ldots, \mc F_r$ of the sequence of flags is given by
\begin{gather*}
 \mc F_1=\{\{x_1=x_2=\cdots=x_n=1\}, \{x_1=x_2=\cdots=x_{n-1}=1\},\ldots,
\{x_1=1\} \} \\
\mc F_2= \{\{x_2=x_3=\cdots=x_n=1\}, \{x_2=x_3=\cdots=x_{n-1}=1\}, \ldots,
\{x_2=1\} \} \\
\cdots \\
\mc F_i=\{
\{x_{i}=x_{i+1}=\cdots=x_n=1\},\{x_{i}=x_{i+1}=\cdots=x_{n-1}=1\},\ldots,
\{x_i=1\} \} \\
\cdots \\
\mc F_r=\{\{x_r=1\}\}.
\end{gather*}
This part of the sequence satisfies condition (2) of Corollary \ref{flblsq}.
Then the easiest way to complete the sequence is to take flags with just one
element beginning with the rank $1$ strata of $\mc D_r^1$ (the only stratum of
rank $0$ is $\{x_1=x_2=\ldots=x_n=1\}$), then the rank $2$ strata and so on.

Now that the sequence of flags exists, Corollary \ref{flblsq}
 ensures that the morphisms in (\ref{BLXnMn}) are well-defined.

Indeed, the usual map $\mnb{r} \ra (\p^1)^r$ which maps $(0,z_1, \ldots,
z_r,1,\infty)$ to $(z_1, \ldots, z_r)$ sends $\Phi_r$ to the standard simplex
$\Delta_r=\{0<t_1<\ldots<t_r<1\}$ and maps $B_r$ to the algebraic
boundary of $\Delta_r$. A first sequence of blow-ups along the subvarieties
$\{0=t_1=\ldots
t_i\}$ corresponds to the change of variables from the simplicial to the
cubical coordinates (\ref{blowup}). In order to recover $B_r$, the blow-up
along the proper transform of the subvarieties $\{t_i=t_{i+1}=\ldots =t_j\}$
and $\{t_i=t_{i+1}=\ldots =t_r=1\}$ still has to be performed. The expression
of these subvarieties in cubical coordinates is given by
$\{x_i=x_{i+1}=\ldots=x_j=1\}$. The fact that it seems that we are blowing 
up less strata in order
to recover $\mnb{r}$ from $(\p^1)^r$ using $\tilde{\delta_r}$ (\ref{BLXnMn})
comes from the fact that we are only looking at $\mnb{r}\sm A_r$.
\end{proof}
From the previous lemma we deduce 
\begin{coro} 
\begin{enumerate}
 \item Let $\mb a=(a_1, \ldots, a_b)$ be a $b$-tuple of integers with $a_1
\geqs 2$ and $a_1+ \cdots + a_b=n+m$. Using the previous convention, we have
the following equality of framed mixed Tate motives
\[
 \zeta^{fr. \mc M}(\mb a,\id)=\left[\HH^{n+m}\pa{\mnb{n+m}\sm A_0 ;
B_{n+m}^{A_0}}; [\om_{\mb a}], [\Phi_{n+m}]\right].
\]
\item Let $\mb k$ and $\mb l$ be as in proposition \ref{stm}, then there exist
two distinguished flags $(\mc F, i_1, \ldots, i_p)$ and $(\mc F',j_1, \ldots, j
_q)$ such that the length of $\mc F$ (resp. $\mc F'$) is $i_p$
(resp. $j_q$) with $i_1, j_1 \geqs 2$ and the sets $I_{i_p}$ and
$I_{j_q}'$ form a partition of 
$\ent{1,n}$. The following equality of framed mixed Tate motives holds
\begin{multline*}
 \zeta^{fr. \mc M}(\mc F, i_1, \ldots i_p | \mc F', j_1, \ldots
j_q)= \\ \left[\HH^{n+m}\pa{\mnb{n+m}\sm A_0, B_{n+m}^{A_0}}; [\oms k \w \oms
l],[\Phi_{n+m}]\right].
\end{multline*}
\end{enumerate}

As a consequence, for all $\sigma \in \on{st}(\mb k, \mb l)$, the framed
mixed Tate motives 
\[\left[\HH^{n+m}\pa{\mnb{n+m}\sm A_0, B_{n+m}^{A_0}};
[\om_{\sigma}],[\Phi_{n+m}]\right]
\]
 are equal to the frame mixed motives 
$\zeta^{fr. \mc M}(\sigma)$.
\end{coro}
\begin{proof}
In 1. and 2., the map on the underlying vector  space is given by
$\alpha_{n+m}^*$ (cf. (\ref{BLXnMn})). As $\wh{C}_{n+m}$ is map to
$\Phi_{n+m}$, knowing the behaviour of $\alpha_{n+m}^*$ with respect to the
form $\om_{\mb a}$ and $\oms k \w \oms l$ is enough to deduce that
$\alpha_{n+m}^*$ respects the frames.
\begin{enumerate}
 \item As the the map $\alpha_{n+m}^{*}$ has no effect on the 
$u_i$ coordinates on 
 $\mnb{n+m}\sm A_0$, we have $\alpha_{n+m}^*(\om_{\mb
a})=\om_{\mb a }^{\id}$, and thus the equality of framed mixed Tate motives.
\item Writing down in cubical coordinates the expression 
\[
\om_{\mb k}=f_{\mb k}(u_1, \ldots , u_n) d^n u\quad \mx{and}\quad \om_{\mb
l}=f_{\mb k}(u_{n+1}, \ldots , u_{n+m}) d^m u
\]
leads to the definition of two distinguished flags
\[
 (\mc F, i_1, \ldots, i_p) \quad \mx{and} \quad(\mc F',j_1, \ldots, j_q),
\]
 as in remark
\ref{flagom} with $s=\id$. The fact that $\alpha_{n+m}^*$ respects the frames
come from the equality
\[
 \om^{\mc F'}_{i_1,\ldots, i_p} \w \om^{\mc F'}_{j_1, \ldots,
j_q}=\alpha_{n+m}^*(\oms k \w \oms l).
\]
\end{enumerate}
\end{proof}
The only thing that remains to be checked to complete the proof of proposition
\ref{stm} is, using the notation of the previous lemma, that 
\[
\zeta^{fr. \mc M}(\mc F, i_1, \ldots i_p | \mc F', j_1, \ldots
j_q)=\sum_{\sigma \in \on{st}(\mb k , \mb l)} \zeta^{fr. \mc M} (\sigma, \id).
\]
Using the computation of section \ref{secstuint}, in particular the
proposition \ref{stuprop}, we have that for each $\sigma \in \on{st}(\mb k
,\mb l)$ there exists a permutation $s_{\sigma}$ such that
\[
 [\om^{\mc F'}_{i_1,\ldots, i_p} \w \om^{\mc F'}_{j_1, \ldots,
j_q}]= \sum_{\sigma \in \on{st}(\mb k
,\mb l)} [\om_{\sigma, s_{\sigma}}].
\]
As the divisor $A_{\mc F', j_1, \cdots, j_q}^{\mc F, i_1, \ldots, i_p}$ of
$\om^{\mc F'}_{i_1,\ldots, i_p} \w \om^{\mc F'}_{j_1, \ldots,
j_q}$ and the divisors $A_{\sigma, s_{\sigma}}$ are in $\wh{A}_{n+m}$, lemma
\ref{inclmot} and an analogue of lemma \ref{idio} show that
\begin{align}\label{stmot1}
 \zeta^{fr. \mc M}(\mc F, i_1, \ldots i_p | \mc F', j_1, \ldots
j_q)=\sum_{\sigma \in \on{st}(\mb k , \mb l)} \zeta^{fr. \mc M} (\sigma,
s_{\sigma}).
\end{align}

Permuting the variables gives a well defined morphism $X_{n+m}\ra
X_{n+m}$ that preserves $\wh{C}_{n+m}$ and its algebraic boundary
$\wh{B}_{n+m}$. It leads, on each term of the RHS of \eqref{stmot1}, to an
equality
\[
 \zeta^{fr. \mc M} (\sigma,
s_{\sigma})=\zeta^{fr. \mc M} (\sigma,\id),
\]
and hence to 
\[
  \zeta^{fr. \mc M}(\mc F, i_1, \ldots i_p | \mc F', j_1, \ldots
j_q)=\sum_{\sigma \in \on{st}(\mb k , \mb l)} \zeta^{fr. \mc M} (\sigma,
\id).
\]
 and Proposition \ref{stm}.

\end{proof}

\bibliographystyle{amsalpha}
\nocite{}
\bibliography{shubib}

\providecommand{\bysame}{\leavevmode\hbox to3em{\hrulefill}\thinspace}
\providecommand{\MR}{\relax\ifhmode\unskip\space\fi MR }
\providecommand{\MRhref}[2]{%
  \href{http://www.ams.org/mathscinet-getitem?mr=#1}{#2}
}
\providecommand{\href}[2]{#2}
\begin{thebibliography}{BGSV90}

\bibitem[BGSV90]{BGSV}
A.~A. Be{\u\i}linson, A.~B. Goncharov, V.~V. Schechtman, and A.~N. Varchenko,
  \emph{Aomoto dilogarithms, mixed {H}odge structures and motivic cohomology of
  pairs of triangles on the plane}, The Grothendieck Festschrift, Vol.\ I,
  Progr. Math., vol.~86, Birkh\"auser Boston, Boston, MA, 1990, pp.~135--172.

\bibitem[Bro06]{Brownphd}
Francis Brown, \emph{Multiple zeta values and periods of moduli spaces
  {$\overline{\mathfrak{M}}_{0,n}(\mathbb{R})$}.}, Ph.D. thesis, Universit{\'e}
  de Bordeaux, ar{X}iv:mmath/0606419, 2006.

\bibitem[GM04]{MZMSGM}
A.~B. Goncharov and Yu.~I. Manin, \emph{Multiple {$\zeta$}-motives and moduli
  spaces {$\overline{\m_{0,n}}$}}, Compos. Math. \textbf{140} (2004), no.~1,
  1--14.

\bibitem[Gon99]{VHMGon}
A.~B. Goncharov, \emph{Volumes of hyperbolic manifolds and mixed {T}ate
  motives}, J. Amer. Math. Soc. \textbf{12} (1999), no.~2, 569--618.

\bibitem[Gon01]{MPMTMGon}
\bysame, \emph{Multiple polylogarithms and mixed tate motives},
  math.AG/0103059, May 2001.

\bibitem[Gon02]{PMMGon}
\bysame, \emph{Periods and mixed motives}, www.arxiv.org/abs/math.AG/0202154,
  May 2002.

\bibitem[Gon05]{GSFGGon}
\bysame, \emph{Galois symmetries of fundamental groupoids and noncommutative
  geometry}, Duke Math. J. \textbf{128} (2005), no.~2, 209--284.

\bibitem[Hu03]{COVHu}
Yi~Hu, \emph{A compactification of open varieties}, Trans. Amer. Math. Soc.
  \textbf{355} (2003), no.~12, 4737--4753 (electronic).

\bibitem[Uly02]{PCCSUly}
Alexander~P. Ulyanov, \emph{Polydiagonal compactification of configuration
  spaces}, J. Algebraic Geom. \textbf{11} (2002), no.~1, 129--159.

\bibitem[Voe00]{Vo00}
V.~Voevodsky, \emph{Triangulated category of motives over a field}, Cycles,
  transfers, and motivic homology theories, Annals of Math. Studies, vol. 143,
  Princeton University Press., 2000.

\end{thebibliography}


\end{document}